\documentclass[12pt, a4paper,reqno]{amsart}
\usepackage{amscd, amssymb, pb-diagram}
\usepackage{palatino,euler}
\usepackage{amsthm} 
\usepackage[margin=2.5cm]{geometry}
\usepackage{tikz}

\allowdisplaybreaks

\def\Aut{\operatorname{Aut}}

\def\newspan{\operatorname{span}}

\def\clsp{\overline{\operatorname{span}}}

\def\Tr{\operatorname{Tr}}

\def\C{\mathbb{C}}

\def\R{\mathbb{R}}
\def\N{\mathbb{N}}
\def\Z{\mathbb{Z}}
\def\T{\mathbb{T}}

\def\Q{\mathbb{Q}}

\def\FF{\mathcal{F}}

\def\HH{\mathcal{H}}

\def\KK{\mathcal{K}}
\def\LL{\mathcal{L}}

\def\NN{\mathcal{N}}
\def\OO{\mathcal{O}}

\def\SS{\mathcal{S}}
\def\TT{\mathcal{T}}
\def\UU{\mathcal{U}}

\newcommand{\dashind}{\operatorname{\!-Ind}}

\renewcommand{\mid}{\,\colon\,}

\def\d{|X|}

\newtheorem{thm}{Theorem}[section]
\newtheorem{cor}[thm]{Corollary}

\newtheorem{lemma}[thm]{Lemma}
\newtheorem{prop}[thm]{Proposition}

\theoremstyle{definition}

\theoremstyle{remark}
\newtheorem{remark}[thm]{Remark}

\newtheorem{example}[thm]{Example}

\numberwithin{equation}{section}

\tikzstyle{vertex}=[circle]
\tikzstyle{goto}=[->,shorten >=1pt,>=stealth,semithick]

\begin{document}

\date{\today}
\title[Equilibrium states and self-similar groups]{Equilibrium states on the Cuntz-Pimsner algebras\\ of self-similar actions}
\author{Marcelo Laca}
\address{Marcelo Laca, Department of Mathematics and Statistics\\
University of Victoria\\
Victoria, BC V8W 3P4\\
Canada}
\email{laca@math.uvic.ca}
\author[Iain Raeburn]{Iain Raeburn}
\address{Iain Raeburn, Department of Mathematics and Statistics, University of Otago, PO Box 56, Dunedin 9054, New Zealand}
\email{iraeburn@maths.otago.ac.nz }
\author[Jacqui Ramagge]{Jacqui Ramagge}
\address{Jacqui Ramagge, School of Mathematics and Applied Statistics, University of Wollongong, NSW 2522, Australia}
\email{ramagge@uow.edu.au}
\author[Michael F. Whittaker]{Michael F. Whittaker}
\address{Michael F. Whittaker, School of Mathematics and Applied Statistics, University of Wollongong, NSW 2522, Australia}
\email{mfwhittaker@gmail.com}

\thanks{This research was supported by the Natural Sciences and Engineering Research Council of Canada, the Marsden Fund of the Royal Society of New Zealand, and the Australian Research Council.}

\begin{abstract}

We consider a family of Cuntz-Pimsner algebras associated to self-similar group actions, and their Toeplitz analogues. Both families carry natural dynamics implemented by automorphic actions of the real line, and we investigate the equilibrium states (the KMS states) for these dynamical systems.

We find that for all inverse temperatures above a critical value, the KMS states on the Toeplitz algebra are given, in a very concrete way, by traces on the full group algebra of the group. At the critical inverse temperature, the KMS states factor through states of the Cuntz-Pimsner algebra; if the self-similar group is contracting, then the Cuntz-Pimsner algebra has only one KMS state. We apply these results to a number of examples, including the self-similar group actions associated to integer dilation matrices, and the canonical self-similar actions of the basilica group and the Grigorchuk group.
\end{abstract}

\maketitle

\section{Introduction}
We study operator-algebraic dynamical systems consisting of an action $\sigma$ of the real line $\R$ on a $C^*$-algebra $B$. Such systems have been used to model time evolution in physics, and there the states are positive functionals on $B$. In models from statistical mechanics, the equilibrium states are time-invariant states which satisfy a commutation relation called the KMS$_\beta$ condition,  where $\beta$ is a real parameter called the inverse temperature \cite{bra-rob}. However, the KMS condition is purely $C^*$-algebraic, and there is a great deal of evidence that the KMS states can be very interesting even when the system $(B,\sigma)$ is not physical. A famous example is the number-theoretic system studied by Bost and Connes \cite{bc}, which exhibits a phase transition like that of a freezing liquid. Their work generated enormous interest in the computation of KMS states for systems of purely mathematical origin (see, for example, \cite{hl, laca98,lr,ln2,cdl}).

In \cite{lrr}, we analysed the KMS states on a family of Exel crossed products associated to self-coverings of the torus $\T^d$. For an integer matrix $A$, the covering map $e^{2\pi i x}\mapsto e^{2\pi iAx}$ induces an endomorphism $\alpha_A$ of $C(\T^d)$ for which there is a natural transfer operator $L$; the Exel crossed product is then, almost by definition \cite{e1,br}, the Cuntz-Pimsner algebra of a Hilbert bimodule $M_L$ over $C(\T^d)$ defined using $\alpha_A$ and $L$. Both the Cuntz-Pimsner algebra $\OO(M_L)$ and the Toeplitz algebra $\TT(M_L)$ carry natural actions $\sigma$ of $\R$. We showed in \cite{lrr} that the system $(\TT(M_L),\sigma)$ has no KMS states for $\beta$ less than a critical value $\beta_c:=\log|\det A|$, and a large simplex of KMS$_\beta$ states for $\beta$ greater than $\beta_c$; when $A$ is a dilation matrix, there is only one KMS state with inverse temperature $\beta_c$, and this state factors through a state of $(\OO(M_L),\sigma)$.

Our analysis in \cite{lrr} exploited the existence of an orthonormal basis for the right Hilbert module $M_L$ \cite{pr,LR2}, which gives a Cuntz family of isometries $\{s_i\}$ in $\OO(M_L)$. The canonical embedding of $C(\T^d)$ gives a unitary representation $u$ of $\Z^d$ in $\OO(M_L)$, and Proposition~3.3 of \cite{lrr} describes a presentation of $\OO(M_L)$ in terms of the $u_n$ and $s_i$. Our present project started when we noticed that Nekrashevych had defined ``Cuntz-Pimsner algebras'' for self-similar groups by specifying a similar presentation \cite{nek_jot, nekra}. In this paper we extend the analysis in \cite{lrr} to cover quite general self-similar groups, with uniqueness at the critical inverse temperature for a class of self-similar actions that includes the contracting ones. 

A self-similar group consists of a group $G$, a finite set $X$, and an action of $G$ on the set $X^*$ of finite words in the alphabet $X$ for which there is a map $(g,x)\mapsto g|_x$ satisfying $g\cdot(xw)=(g\cdot x)(g|_x\cdot w)$ for $w\in X^*$ (see \S\ref{sec:SSA}). Each integer matrix $A$ gives a self-similar group $(\Z^d,\Sigma)$ in which $\Sigma$ is a set of coset representatives for $\Z^d/A^t\Z^d$ (see \S\ref{dilation_example}), but there are many more: indeed, self-similar groups have been a fertile source of interesting examples for infinite group theory (see \cite{nek_book}, for example). 

For each self-similar group $(G,X)$, we construct a Hilbert bimodule $M$ over the group $C^*$-algebra $C^*(G)$ such that the right module has an orthonormal basis $\{e_x:x\in X\}$ and the left action of $C^*(G)=\clsp\{\delta_g\}$ satisfies $\delta_g\cdot e_x=e_{g\cdot x}\cdot\delta_{g|_x}$. This bimodule has  a Toeplitz algebra $\TT(M)$ and a Cuntz-Pimsner algebra $\OO(M)$, and both carry canonical actions $\sigma$ of $\R$. The Cuntz-Pimsner algebra is the same as that of Nekrashevych \cite{nekra}, but the Toeplitz algebra appears to be new. As previous studies in this general area have consistently showed \cite{el,ln,lr,lrr,hlrs}, the Toeplitz system $(\TT(M),\sigma)$ has a much richer supply of KMS states.

As in \cite{lrr}, there is a critical inverse temperature $\beta_c:=\log|X|$ such that $(\TT(M),\sigma)$ has no KMS states for $\beta$ less than $\beta_c$. For $\beta$ larger than $\beta_c$, we show that the KMS$_\beta$ states of $(\TT(M),\sigma)$ are parametrised by the normalised traces on $C^*(G)$, and we give a formula for the values of these states on a set of elements which span a dense subalgebra of $\TT(M)$ (Theorems~\ref{Conditioning} and~\ref{repn_hilbert}). When the restrictions $g|_v$ of each fixed $g$ form a finite set (see \S\ref{sec:SSA} for details), there is a unique KMS$_{\beta_c}$ state on $(\TT(M),\sigma)$, and it is the only KMS state of $(\TT(M),\sigma)$ which factors through a state of $(\OO(M),\sigma)$ (Theorem~\ref{KMSatcritical}). We do not have an explicit formula for the values of this last state, but we describe a combinatorial procedure for computing its value on a particular generator, and illustrate this procedure in some examples (see \S\ref{sec:Mooreuseful}).

Since we suspect that many  operator algebraists are not familiar  with self-similar group actions, we begin in \S\ref{sec:SSA} with a review of their basic properties. We then discuss some key examples, including odometers, actions of $\Z^d$  associated to integer matrices, and two nonabelian groups called the basilica group and the Grigorchuk group. We then construct our Hilbert bimodule $M$ over $C^*(G)$, and describe presentations of the Toeplitz algebra $\TT(M)$ (Proposition~\ref{Toeplitz}) and the Cuntz-Pimsner algebra $\OO(M)$ (Corollary~\ref{CPquotient}). 

Our computation of KMS states for $\beta>\beta_c$ follows the general program developed in \cite{lr}, \cite{lrr} and \cite{hlrs}. We first find an easily verified relation which allows us to recognise KMS states (Proposition~\ref{KMS>beta}). We then prove existence of KMS states using representation-theoretic methods (Theorem~\ref{repn_hilbert}). As in \cite{lrr}, our construction uses induced representations, but in the setting of self-similar groups, we can use the bimodule $M$ and ideas from \cite{ln} involving Rieffel induction to get a more systematic approach. We prove surjectivity of our parametrisation in \S\ref{parameterise}, by showing  that KMS states are characterised by their conditioning to a small corner in $\TT(M)$. In \S\ref{KMS_CP_alg}, we discuss KMS states on the Cuntz-Pimsner algebra, and then we close with a section on examples.

\section{Self-similar actions}\label{sec:SSA}

If $X$ is a set, we write $X^n$ for the set of words of length $n$ in~$X$, with $X^0=\{\varnothing\}$, and $X^*:=\bigcup_{n=0}^\infty X^n$. A \emph{self-similar action} $(G,X)$ consists of a finite set $X$ and a faithful action of a group $G$ on $X^*$ such that, for all $g\in G$ and $x\in X$, there exist unique $y\in X$ and $h\in G$ such that
\begin{equation}\label{self-similarity}
g\cdot (xw) = y (h\cdot w) \quad \text{ for all } w\in X^*.
\end{equation}
We also assume that $g\cdot\varnothing=\varnothing$, and then taking $w=\varnothing$ shows that $y=g\cdot x$. We call $h$ the {\em restriction} of $g$ to $x$ and denote it by $g|_x$. Thus \eqref{self-similarity} becomes
\[
g \cdot (xw) = (g \cdot x)(g|_x \cdot w) \quad \text{for all } w \in X^*. 
\] 
Then for $g \in G$ and $w=w_1w_2\cdots w_n$ in $X^n$, we have
\[
g\cdot w= (g \cdot w_1) (g|_{w_1} \cdot (w_2 \cdots w_n)) 
= \cdots
= (g\cdot w_1)(g|_{w_1} \cdot w_2) \cdots (g|_{w_1}|_{w_2\cdots}|_{w_{n-1}} \cdot w_n),
\]
and in particular $g\cdot w\in X^n$.

\begin{lemma}[{\cite[\S1.3]{nek_book}}] \label{lem:props}
Suppose $(G,X)$ is a self-similar action.
\begin{enumerate}
\item For each $(g,v)\in G\times X^n$, there exists $g|_v\in G$ satisfying
\begin{equation}\label{self-similarity2}
g\cdot (vw)=(g\cdot v) (g|_v\cdot w) \quad \text{ for all } w\in X^*.
\end{equation}
\item For  $g, h\in G$ and $v,w\in X^*$, we have
\[
g|_{vw}=(g|_{v})|_{w}, \quad
gh|_{v} = g|_{h\cdot v} h|_{v}, \quad\text{and}\quad
g|_{v}^{-1} = g^{-1}|_{g\cdot v}.
\]
\item For every $g \in G$, the map $g:X^n \to X^n$ is bijective.
\end{enumerate}
\end{lemma}

Suppose that $(G,X)$ is a self-similar action. We can view $X^*$ as the vertices of a rooted tree $T_X$ with root $\varnothing$ and edges from $w\to wx$, and then \eqref{self-similarity2} implies that $G$ acts on $T_X$ by graph automorphisms. Indeed, since the action is faithful, the action gives an embedding of $G$ in the automorphism group $\Aut T_X$. Many of the important examples are constructed by specifying $X$ and the subgroup of  $\Aut T_X$.

A self-similar action $(G,X)$ is \emph{finite-state} if for every $g\in G\setminus\{e\}$, the set $\{g|_v:v\in X^*\}$ is finite \cite[page~11]{nek_book}.
As in \cite[\S2.11]{nek_book}, $(G,X)$ is {\em contracting} if there is a finite subset $S$ of $G$ such that for every $g\in G$ there exists $n$ with $\{ g|_{v}: v \in X^*, |v| \geq n\}\subset S$; the smallest such set
\begin{equation}\label{defnuc} 
\NN:= \bigcup_{g \in G} \bigcap_{n=0}^\infty \{ g|_{v}: v \in X^*, |v| \geq n\} 
\end{equation}
is then called the \emph{nucleus} of $(G,X)$. 

Suppose that $(G,X)$ is a self-similar action and $S$ is a subset of $G$ that is closed under restriction. The {\em Moore diagram} of $S$ is the labelled directed graph with vertex set $E^0=S$ and a directed edge from $g$ to $g|_x$ labelled $(x,g \cdot x)$ for each $x \in X$. So an edge
\begin{center}
\begin{tikzpicture}
\node[vertex] (vertexa) at (-2,0)   {$g$};	
\node[vertex] (vertexb) at (0,0)  {$h$}
	edge [<-,>=latex,out=180,in=0,thick] node[auto,swap,pos=0.5]{$\scriptstyle(x,y)$}(vertexa);
\end{tikzpicture}
\end{center}
in the Moore diagram encodes the self-similar relation $g\cdot(xw)=y(h\cdot w)$.

We are particularly interested in the Moore diagram of the nucleus, and will use Moore diagrams to help find the nucleus. Later, we will use larger Moore diagrams to compute values of KMS states.

\begin{prop}\label{Moore_nucleus}
Suppose $(G,X)$ is a self-similar action and $S$ is a subset of $G$ that is closed under restriction. Every vertex in the Moore diagram of $S$ that can be reached from a cycle belongs to the nucleus.
\end{prop}

\begin{proof}
Suppose $g \in G$ is a vertex in the Moore diagram of $S$, and there is a cycle of length $n\geq 1$ consisting of edges labelled $(x_1,y_1), (x_2,y_2), \cdots, (x_n,y_n)$ with $s(x_1,y_1)=g$, $r(x_i,y_i)=s(x_{i+1},y_{i+1})$, and $r(x_n,y_n)=g$. By definition of the Moore diagram we have $g \cdot (x_1 \cdots x_n) = y_1 \cdots y_n$ and $g|_{x_1 \cdots x_n} = g$. Thus $g = g|_{(x_1 \cdots x_n)^m}$ for all $m \in \N$ and
\[
g \in \bigcap_{n \geq 0} \{ g|_{v} :v \in X^*, |v| \geq n\} \implies g \in \bigcup_{h \in G} \bigcap_{n \geq 0} \{ h|_{v} : v \in X^*, |v| \geq n\} = \NN.
\]
A similar argument shows that if $g$ can be reached from a cycle, then there are arbitrarily long paths ending at $g$.
\end{proof}

In the rest of this section, we discuss some key examples of self-similar actions.
\begin{figure}
\begin{tikzpicture}
\node at (0,0) {$\scriptstyle e$};
\node[vertex] (vertexe) at (0,0)   {$\quad$}
	edge [->,>=latex,out=130,in=100,loop,thick] node[auto,pos=0.56]{$\scriptstyle (0,0)$} (vertexe)
	edge [->,>=latex,out=50,in=80,loop,thick] node[auto,swap,pos=0.56]{$\scriptstyle (1,1)$} (vertexe)
	edge [->,>=latex,out=230,in=260,loop,thick] node[auto,swap,pos=0.56]{$\scriptstyle (2,2)$} (vertexe)
	edge [->,>=latex,out=310,in=280,loop,thick] node[auto,pos=0.56]{$\scriptstyle (3,3)$} (vertexe);
\node at (3,0) {$\scriptstyle g$};
\node[vertex] (vertexa) at (3,0)  {$\quad$}
	edge [->,>=latex,out=150,in=30,thick] node[auto,swap,pos=0.5]{$\scriptstyle(0,1)$} (vertexe)
	edge [->,>=latex,out=180,in=0,thick] node[auto,swap,yshift=-0.1cm,pos=0.5]{$\scriptstyle(1,2)$} (vertexe)
	edge [->,>=latex,out=210,in=330,thick] node[auto,pos=0.5]{$\scriptstyle(2,3)$} (vertexe)
	edge [->,>=latex,out=40,in=320,loop,thick] node[auto,pos=0.5]{$\scriptstyle (3,0)$} (vertexg);
\node at (-3,0) {$\scriptstyle {g^{-1}}$};
\node[vertex] (vertex-a) at (-3,0)   {$\quad$}
	edge [->,>=latex,out=30,in=150,thick] node[auto,pos=0.49]{$\scriptstyle(1,0)$} (vertexe)
	edge [->,>=latex,out=0,in=180,thick] node[auto,yshift=-0.1cm,pos=0.5]{$\scriptstyle(2,1)$} (vertexe)
	edge [->,>=latex,out=330,in=210,thick] node[auto,swap,pos=0.5]{$\scriptstyle(3,2)$} (vertexe)
	edge [->,>=latex,out=140,in=220,loop,thick] node[auto,swap,pos=0.5]{$\scriptstyle (0,3)$} (vertex-a);
\end{tikzpicture}
\caption{The Moore diagram for the nucleus of the odometer with $N=4$.}
\label{fig:odometer}
\end{figure}
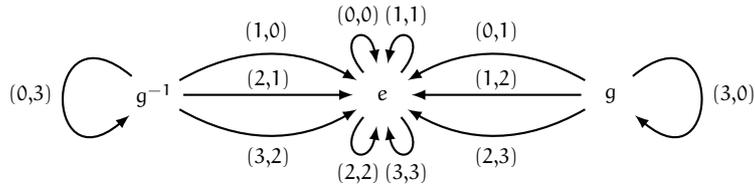
\subsection{Odometers} Fix an integer $N>1$, and let $X_N=\{0,1,\cdots,N-1\}$. We consider the multiplicative free abelian group $G$ with generator $g$, so that $G=\{g^k:k\in \Z\}$. We define an action of $G$ on $X_N^*$ by
\[
g\cdot v=\begin{cases}
(v_1+1)v_2\cdots v_{|v|}&\text{if $v_1<N-1$}\\
0\cdots 0(v_k+1)v_{k+1}\cdots v_{|v|}&\text{if $v_1=\cdots =v_{k-1}=N-1$ and $v_k<N-1$.}
\end{cases}
\]
Then $(G,X_N)$ is a self-similar action with $g|_i=e$ for $i<N-1$ and $g|_{N-1}=g$. This action is contracting with nucleus $\NN=\{e,g,g^{-1}\}$: indeed, if $k>0$ and $|w|>\log_Nk$, then $g^k|_w$ is either $e$ or $g$, and if $k<0$ and $|w|>\log_N|k|$, then $g^k|_w$ is either $e$ or $g^{-1}$. The Moore diagram for $\NN$ for $N=4$ is shown in Figure~\ref{fig:odometer}.

The self-similar action $(G,X_N)$ is called an \emph{odometer}. To see why, identify $X_N^n$ with $\{0,1,\cdots N^n-1\}$ by sending $v$ to $\sum_{i=1}^nv_iN^{i-1}$, and then the action of $g$ on $X_N^n$ adds $1 \pmod{N^n}$.

\subsection{Integer matrices}\label{dilation_example}

Suppose that $A\in M_d(\Z)$ has $N:=|\det A|> 1$, and write $B:=A^t$ for its transpose. We choose a set $\Sigma$ of coset representatives for the quotient $\Z^d/B\Z^d$, and we assume that $0\in \Sigma$. For $n\in \Z^d$, we write $c(n)$ for the representative of $n+B\Z^d$ in $\Sigma$. We note that $\det B=\det A=N$, and hence $\Sigma$ has cardinality $N$.

We now fix an integer $k\geq 1$. Then the set $\Sigma^k$ gives a parametrisation
\[
\big\{b_k(w)=w_1+Bw_2+\cdots B^{k-1}w_k+B^k\Z^d:w\in \Sigma^k\big\}
\]
of $\Z^d/B^k\Z^d$. The following straightforward lemma tells us how the different bijections $b_k$ combine.

\begin{lemma}\label{calcbk+1}
Write $B$ for the injective homomorphism $B:\Z^d/B^k\Z^d\to \Z^d/B^{k+1}\Z^d$ which takes $m+B^k\Z^d$ to $Bm+B^{k+1}\Z^d$. Then for $x\in \Sigma$ and $w\in \Sigma^k$ we have $b_{k+1}(xw)=x+B(b_k(w))$.
\end{lemma}

\begin{prop}\label{prop:dilation}
The actions of the additive abelian group $\Z^d$ on its quotients $\Z^d/B^k\Z^d$ combine to give an action of $\Z^d$ on $\Sigma^*$ such that
\begin{equation}\label{defaction}
n\cdot w=b_k^{-1}(n\cdot b_k(w))\quad \text{for $k\geq 1$ and $w\in \Sigma^k$.}
\end{equation}
The pair $(\Z^d,\Sigma)$ is a self-similar action, and for $n\in \Z^d$, $x\in \Sigma$ we have
\begin{equation}\label{idrest}
n\cdot x=c(n+x)\quad\text{and}\quad n|_{x}=B^{-1}(n+x-c(n+x)).
\end{equation}
If $A$ is a dilation matrix (in the sense that all its complex eigenvalues $\lambda$ satisfy $|\lambda|>1$), then $(\Z^d,\Sigma)$ is contracting.
\end{prop}

\begin{proof}
Since 
\begin{align*}
(m+n)\cdot(p+B^k\Z^d)&=(m+n)+p+B^k\Z^d=m\cdot(n+p+B^k\Z^d)\\
&=m\cdot(n\cdot(p+B^k\Z^d)),
\end{align*}
the formula \eqref{defaction} gives an action of the additive group $\Z^d$ on $\Sigma^k$. To establish \eqref{idrest}, we take $x\in \Sigma$, $w\in\Sigma^k$, and compute:
\begin{align*}
b_{k+1}(n\cdot(xw))&=n+x+Bw_1\cdots +B^kw_k+B^{k+1}\Z^d\\
&=c(n+x)+(n+x-c(n+x))+Bw_1\cdots +B^kw_k+B^{k+1}\Z^d\\
&=c(n+x)+B\big(B^{-1}(n+x-c(n+x))+w_1\cdots +B^{k-1}w_k+B^{k}\Z^d),\\
&=c(n+x)+B\big(b_k(B^{-1}(n+x-c(n+x))\cdot w)\big),
\end{align*}
which by Lemma~\ref{calcbk+1} is $b_{k+1}\big(c(n+x)(B^{-1}(n+x-c(n+x))\cdot w)\big)$. Thus 
\[
n\cdot(xw)=c(n+x)(B^{-1}(n+x-c(n+x))\cdot w),
\]
which implies that $(\Z^d,\Sigma)$ is self-similar, and gives \eqref{idrest}.

Now we suppose that $A$ is a dilation matrix. For $x\in \Sigma$, the virtual endomorphism $\phi_x$ associated to $x\in X$ (as in \cite[\S2.5]{nek_book}) is the map $n\mapsto n|_x$ from the stabiliser of $x$ into $\Z^d$; the stabiliser is $B\Z^d$, and for $n\in B\Z^d$, $\phi_x(n)=B^{-1}n$. Thus the linear transformation $\Q\otimes \phi:\Q^d\to \Q^d$ considered in \cite[Theorem~2.12.1]{nek_book} has matrix $B^{-1}$. Since $\det B=\det A\not=0$, the eigenvalues of $B^{-1}$ are the inverses $\lambda^{-1}$ of the eigenvalues of $B$. Since $A$ is a dilation matrix, we have $|\lambda^{-1}|<1$ for all such $\lambda$, and $B^{-1}$ has spectral radius $\rho(B^{-1})<1$. Thus \cite[Theorem~2.12.1]{nek_book} implies that $(\Z^d,\Sigma)$ is contracting. 
\end{proof}

\begin{figure}
\begin{tikzpicture}
\node at (0,0) {$\scriptstyle e$};
\node[vertex] (vertexe) at (0,0)   {$\,$}
	edge [->,>=latex,out=310,in=350,loop,thick] node[auto,swap,xshift=-0.1cm,yshift=0.1cm,pos=0.5]{$\scriptstyle (y,y)$} (vertexe)
	edge [->,>=latex,out=230,in=190,loop,thick] node[auto,xshift=0.1cm,yshift=0.1cm,pos=0.5]{$\scriptstyle (x,x)$} (vertexe);
\node at (-1.2856,1.532) {$\scriptstyle b$};
\node[vertex] (vertexb) at (-1.2856,1.532)   {$\,$}
	edge [->,>=latex,out=330,in=140,thick] node[auto,xshift=-0.2cm,pos=0.3]{$\scriptstyle(y,y)$} (vertexe);
\node at (-1.97,0.347) {$\scriptstyle a$};
\node[vertex] (vertexa) at (-1.97,0.347)  {$\,$}
	edge [->,>=latex,out=330,in=160,thick] node[auto,swap,yshift=0.1cm,pos=0.3]{$\scriptstyle(y,x)$} (vertexe)
	edge [->,>=latex,out=100,in=200,thick] node[auto,xshift=0.1cm,pos=0.4]{$\scriptstyle(x,y)$}(vertexb)
	edge [<-,>=latex,out=20,in=280,thick] node[auto,swap,xshift=-0.1cm,pos=0.6]{$\scriptstyle(x,x)$}(vertexb);
\node at (1.2856,1.532) {$\scriptstyle \,\,\, b^{-1}$};
\node[vertex] (vertex-b) at (1.2856,1.532)   {$\,$}
	edge [->,>=latex,out=210,in=40,thick] node[auto,swap,xshift=0.2cm,pos=0.3]{$\scriptstyle(y,y)$} (vertexe);
\node at (1.97,0.347) {$\scriptstyle \,\,\, a^{-1}$};
\node[vertex] (vertex-a) at (1.97,0.347)  {$\,$}
	edge [->,>=latex,out=210,in=20,thick] node[auto,yshift=0.1cm,pos=0.3]{$\scriptstyle(x,y)$} (vertexe)
	edge [->,>=latex,out=80,in=-20,thick] node[auto,swap,xshift=-0.1cm,pos=0.4]{$\scriptstyle(y,x)$}(vertex-b)
	edge [<-,>=latex,out=160,in=260,thick] node[auto,xshift=0.15cm,pos=0.6]{$\scriptstyle(x,x)$}(vertex-b);
\node at (-0.684,-1.879) {$\scriptstyle ab^{-1}$};
\node[vertex] (vertexa-b) at (-0.684,-1.879)   {$\,$}
	edge [->,>=latex,out=90,in=260,thick] node[auto,xshift=0.1cm,pos=0.1]{$\scriptstyle(y,x)$} (vertexe);
\node at (0.684,-1.879) {$\scriptstyle ba^{-1}$};
\node[vertex] (vertexb-a) at (0.684,-1.879)  {$\,$}
	edge [->,>=latex,out=90,in=280,thick] node[auto,xshift=-0.1cm,pos=0.1,swap]{$\scriptstyle(x,y)$} (vertexe)
	edge [->,>=latex,out=140,in=40,thick] node[auto,swap]{$\scriptstyle(y,x)$}(vertexa-b)
	edge [<-,>=latex,out=220,in=320,thick] node[auto]{$\scriptstyle(x,y)$}(vertexa-b);
\end{tikzpicture}
\caption{The Moore diagram for the nucleus of the basilica group.}
\label{fig:Basilica}
\end{figure}
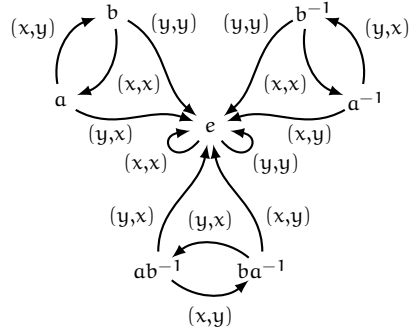
\subsection{The basilica group}\label{Basilica}
Let $X$ be the set $\{x,y\}$ with $|X|=2$, and consider the rooted homogeneous tree $T_X$ with vertex set $X^*$. We recursively define two automorphisms $a$ and $b$ of $T_X$ by
\begin{alignat}{2}\label{basilica_action}
a \cdot(xw)&= y (b\cdot w) & \qquad   a \cdot (yw)&= x w  \\
b \cdot(xw)&= x (a \cdot w)& \qquad   b \cdot (yw)&= y w \notag
\end{alignat}
for $w \in X^*$. Then the \emph{basilica group} $B$ is the subgroup of $\Aut T_X$ generated by $\{a,b\}$. The pair $(B,X)$ is then a self-similar action. 

We now show that the basilica action $(B,X)$ is contracting and compute the nucleus. This is probably well-known, but since our answer seems to contradict an assertion in \cite[page~111]{nek_book}, we give a detailed proof.

\begin{prop}\label{Bascontract}
The basilica group action $(B,X)$ is contracting, with nucleus
\[
\NN = \{e,a,b,a^{-1},b^{-1},ab^{-1},ba^{-1}\}; 
\]
the Moore diagram of $\NN$ is in Figure~\ref{fig:Basilica}.
\end{prop}

\begin{proof}
Let $\SS=\{e,a,b,a^{-1},b^{-1},ab^{-1},ba^{-1}\}$, for which we have the Moore diagram in Figure~\ref{fig:Basilica}.
Since this Moore diagram has a cycle at every vertex, Proposition~\ref{Moore_nucleus} implies that $\SS\subset\NN$. So we have to prove $\NN\subset\SS$,

We claim that every $g \in B\setminus\{ab^{-1},ba^{-1}\}$ which can be written as a reduced product of two elements of $T:=\{a,b,a^{-1},b^{-1}\}$ has the property that $g|_v \in T\cup\{e\}$ for every word $v \in X^2$. There are 12 non-trivial length-two words in $G$; we delete $ab^{-1},ba^{-1}$ from the list, and compute the other restrictions to $x$ and $y$. We find that
\begin{alignat*}{4}
a^2|_{x}&=b &\quad a^2|_{y}&=b &\qquad ab|_{x}&=ba &\quad ab|_{y}&=e \\
a^{-2}|_{x}&=b^{-1} &\quad a^{-2}|_{y}&=b^{-1} &\qquad a^{-1}b|_x&=a&\quad a^{-1}b|_y&=b^{-1}\\
a^{-1}b^{-1}|_{x}&=a^{-1}&\quad a^{-1}b^{-1}|_{y}&=b^{-1}&\qquad ba|_{x}&=b &\quad ba|_{y}&=a  \\
b^2|_{x}&=a^2 &\quad b^2|_{y}&=e &\qquad b^{-1}a|_{x}&=b &\quad b^{-1}a|_{y}&=a^{-1} \\
b^{-1}a^{-1}|_{x}&=e  &\quad b^{-1}a^{-1}|_{y}&=a^{-1}b^{-1} &\qquad b^{-2}|_{x}&=a^{-2} &\quad b^{-2}|_{y}&=e.
\end{alignat*}  
Now we observe that further restrictions of $b^{-1}a^{-1}$, $ab$, $b^2$ and $b^{-2}$ are all in $T\cup\{e\}$, and we have justified our claim.

Next we suppose that $n\geq 3$ and that $g$ can be written as a reduced product of $n$ elements of $T$, and take $v\in X^2$. We claim that $g|_v$ can be written as a product of at most $n-1$ elements of $T$. We factor off the last two elements of $T$, say $g=g'h$. Since $g|_v=g'|_{h\cdot v}h|_v$, the claim in the previous paragraph implies that, unless $(h,v)=(ab^{-1},xy)$ or $(ba^{-1},yx)$, we have  $h|_v\in T$. Since $g'|_{h\cdot v}$ is a product of at most $n-2$ elements of $T$, we have $g|_v=g'|_{h\cdot v}h|_v$ written as a product of $n-1$ elements of $T$. So we have to deal with $(h,v)=(ab^{-1},xy)$ and $h=(ba^{-1},yx)$. Now, since $g'$ is a product of $n-2$ elements of $T$ and $n\geq 3$, we can pull one element $t$ out of $g'$, and it suffices for us to prove that $(tab^{-1})|_{xy}$ and $(tba^{-1})|_{yx}$ can be written as products of two elements of $T$. We compute:
\begin{align*}
(a^2b^{-1})|_x&=a|_{(ab^{-1})\cdot x}(ab^{-1})|_x=a|_yba^{-1}=ba^{-1}\\
(b^{-1}ab^{-1})|_x&=b^{-1}|_y(ba^{-1})=ba^{-1}\\
(bab^{-1})|_{x}&=b|_{y}ba^{-1}=ba^{-1}\\
(a^{-1}ba^{-1}|_y&=a^{-1}|_xab^{-1}=ab^{-1}\\
(b^2a^{-1})|_{yx}&=(b|_{x}(ab^{-1}))|_x=(a(ab^{-1}))|_x=a|_yba^{-1}=ba^{-1}\\
(aba^{-1})|_{yx}&=(a|_xab^{-1})|_x=(bab^{-1})|_x=b|_yba^{-1}=ba^{-1}.
\end{align*}
This completes the proof of the claim.

Successive applications of the claim in the previous paragraph show that if $g$ is a product of $n$ elements of $T$ and $n\geq 3$, then $g|_v$ is a product of at most $2$ elements for every $v$ with $|v|\geq 2(n-2)$. Now the calculations in the first paragraph show that a further restriction to a word in $X^2$ gets us into $\SS$. Thus for $|v|\geq 2(n-1)$, we have $g|_v\in\SS$. So the inside intersection in \eqref{defnuc} is contained in $\SS$, and so is $\NN$.
\end{proof}

Our next proposition says that $B$ has a large abelian quotient. We believe this is known, but we do not know where a proof has been published. 

We want to use the presentation of $B$ found by Bartholdi and Vir\'{a}g \cite[Lemma~11]{BV}, building on work of Grigorchuk and \.Zuk \cite{GZ1}. We consider the collection $Y^*$ of all nonempty words in $Y:=\{a,b,a^{-1},b^{-1}\}$, and the transformation $\sigma:Y^*\to Y^*$ which replaces every appearance of $a$ by $bb$, every $a^{-1}$ by $b^{-1}b^{-1}$, every $b$ by $a$, and every $b^{-1}$ by $a^{-1}$. For $c,d\in Y^*$, we write $c^{-1}$ for the word obtained by formally inverting $c$, and $[c,d]$ for the word $c^{-1}d^{-1}cd$. Then \cite[Lemma~11]{BV} says that $B$ has the presentation
\begin{equation}\label{presB}
B=\langle a,b : \sigma^n([a,b^{-1}ab])=e \text{ for all } n \in \N \rangle.
\end{equation}

\begin{prop}\label{quotB}
Let $[B,B]$ be the commutator subgroup of the basilica group $B$, and let $q:B\to B/[B,B]$ be the quotient map. Then there is an isomorphism $\phi$ of $B/[B,B]$ onto $\Z^2$ such that $\phi(q(a))=(1,0)$ and $\phi(q(b))=(0,1)$.
\end{prop}

\begin{proof}
Since $a':=(1,0)$ and $b':=(0,1)$ commute, we have $\sigma^n([a',(b')^{-1}a'b'])=e$ for all $n$. Thus there is a homomorphism of $B$ into $\Z^2$ taking $\{a,b\}$ to $\{a',b'\}$, and this factors through a homomorphism $\phi:B/[B,B]\to\Z^2$. 

Since $[a,b]$ and $[a,b^{-1}]$ belong to the commutator subgroup, $q(a)$ commutes with $q(b)$ and $q(b^{-1})$, and hence for every $g\in B=\langle a,b\rangle$, there are $k,l\in \Z$ such that $q(g)=q(a^kb^l)$. Since $\phi(q(a^kb^l))=(k,l)$, we deduce both that $\phi$ is surjective and that $\phi$ is injective.
\end{proof}

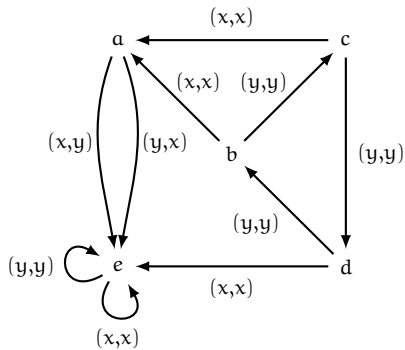
\begin{figure}
\begin{tikzpicture}
\node at (0,0) {$\scriptstyle e$};
\node[vertex] (vertexe) at (0,0)   {$\,$}
	edge [->,>=latex,out=210,in=150,loop,thick] node[auto,pos=0.5]{$\scriptstyle (y,y)$} (vertexe)
	edge [->,>=latex,out=240,in=300,loop,thick] node[auto,swap,pos=0.5]{$\scriptstyle (x,x)$} (vertexe);
\node at (1.5,1.5) {$\scriptstyle b$};
\node[vertex] (vertexb) at (1.5,1.5)   {$\,$};
\node at (0,3) {$\scriptstyle a$};
\node[vertex] (vertexa) at (0,3)  {$\,$}
	edge [->,>=latex,out=250,in=100,thick] node[auto,swap,xshift=0.1cm,pos=0.45]{$\scriptstyle(x,y)$} (vertexe)
	edge [->,>=latex,out=290,in=80,thick] node[auto,xshift=-0.1cm,pos=0.45]{$\scriptstyle(y,x)$}(vertexe)
	edge [<-,>=latex,out=315,in=135,thick] node[auto,xshift=-0.25cm,pos=0.6]{$\scriptstyle(x,x)$} (vertexb);
\node at (3,3) {$\scriptstyle c$};
\node[vertex] (vertexc) at (3,3)   {$\,$}
	edge [<-,>=latex,out=225,in=45,thick] node[auto,swap,xshift=0.25cm,pos=0.6]{$\scriptstyle(y,y)$} (vertexb)
	edge [->,>=latex,out=180,in=0,thick] node[auto,swap,pos=0.5]{$\scriptstyle(x,x)$} (vertexa);
\node at (3,0) {$\scriptstyle d$};
\node[vertex] (vertexd) at (3,0)  {$\,$}
	edge [->,>=latex,out=180,in=0,thick] node[auto,pos=0.5]{$\scriptstyle(x,x)$} (vertexe)
	edge [<-,>=latex,out=90,in=270,thick] node[auto,swap,pos=0.5]{$\scriptstyle(y,y)$}(vertexc)
	edge [->,>=latex,out=135,in=315,thick] node[auto,xshift=0.15cm,pos=0.6]{$\scriptstyle(y,y)$} (vertexb);
\end{tikzpicture}
\caption{The Moore diagram for the nucleus of the Grigorchuk group.}
\label{fig:Grig}
\end{figure}
\subsection{The Grigorchuk group}\label{Grigorchuk}
We again consider the set $X=\{x,y\}$ and the associated rooted tree $T_X$ with vertex set $X^*$. We define automorphisms $a$, $b$, $c$, and $d$ of $T_X$ recursively by
\begin{alignat}{2}\label{grig_action}
a \cdot(xw)&= yw & \qquad a \cdot (yw)&= xw \\
b \cdot (xw)&= x(a \cdot w) & \qquad  b \cdot(yw)&= y(c \cdot w)\notag \\
c \cdot (xw)&= x(a \cdot w) & \qquad   c \cdot (yw)&= y(d \cdot w)\notag\\
d \cdot (xw)&= xw & \qquad   d \cdot (yw)&= y(b \cdot w).\notag
\end{alignat}
Then the Grigorchuk group $G$ is the subgroup of $\Aut T_X$ generated by $\{a,b,c,d\}$. 

The first assertions of the next proposition are also in the proof of \cite[Theorem~1.6.1]{nek_book}; the assertion about the nucleus is stated without proof on page~57 of \cite{nek_book}.

\begin{prop}\label{grig:nuc}
The generators $a$, $b$, $c$, $d$ of $G$ all have order two, and satisfy $cd=b=dc$, $db=c=bd$ and $bc=d=cb$. The self-similar action $(G,X)$ is contracting with nucleus $\NN = \{e,a,b,c,d\}$.
\end{prop}

\begin{proof}
The first two relations in \eqref{grig_action} imply that $a^2=e$. Now the other relations imply that
\begin{alignat*}{2}
b^2 \cdot (xw)&= x(a^2 \cdot w) =xw& \qquad   b^2 \cdot(yw)&= y(c^2 \cdot w) \\
c^2\cdot (xw)&= x(a^2\cdot w)=xw & \qquad   c^2 \cdot (yw)&= y(d^2 \cdot w) \\
d^2 \cdot (xw)&= xw & \qquad   d^2 \cdot (yw)&= y(b^2 \cdot w),
\end{alignat*}
and we can prove by induction on $n=|v|$ that $b^2\cdot v=c^2\cdot v=d^2\cdot v=v$ for all $v\in X^*$. Thus $b^2=c^2=d^2=e$ in $G\subset \Aut T_X$. In particular, every element of $G$ is a product of generators $\{a,b,c,d\}$.

Next we note that $a$ is determined by the first two relations in \eqref{grig_action}, and then the other six determine $(b,c,d)$. A computation shows that $(b',c',d')=(cd,db,bc)$ satisfies the same six recurrence relations, and hence we have $cd=b$, $db=c$ and $bc=d$. Since the generators all have order two, inverting gives $dc=b$, $bd=c$ and $cb=d$. Thus the only elements of $G$ which are products of two generators are the elements of
\[
R:=\{ab, ba, ac, ca, ad, da\}. 
\]
Twelve calculations show that for every $g\in R$, both $g|_x$ and $g|_y$ belong to $\{e,a,b,c,d\}$. Thus if $g$ is a product of $n$ generators, we have $g|_v\in \{e,a,b,c,d\}$ for every word $v$ with $|v|\geq n-1$. This proves that $(G,X)$ is contracting, and that the nucleus is contained in $\{e,a,b,c,d\}$.

Since every vertex in the Moore diagram of $\{e,a,b,c,d\}$ in Figure~\ref{fig:Grig} can be reached from a cycle, Proposition~\ref{Moore_nucleus} implies that $\{e,a,b,c,d\}$ is contained in the nucleus.
\end{proof}

\section{Universal algebras associated to a self-similar action}\label{sec:Toeplitz}

Suppose that $(G,X)$ is a self-similar action, and let $C^*(G)$ be the full group $C^*$-algebra of $G$ generated by the unitary representation $\{\delta_g:g\in G\}$. We are interested in two $C^*$-algebras associated to $(G,X)$, which we construct as the Toeplitz algebra and the Cuntz-Pimsner algebra of a Hilbert bimodule $M$ over $C^*(G)$.

As a right Hilbert $C^*(G)$-module, $M$ is the direct sum $M=\bigoplus_{x\in X}C^*(G)$; thus $M=\{m=(m_x)_{x\in X} :m_x\in C^*(G)\}$, with module action $(m_x)\cdot a=(m_xa)$ and inner product 
\[
\langle m,n\rangle=\sum_{x\in X} m_x^*n_x.
\]
For $y\in X$ we define $e_y\in M$ by
\[
(e_y)_x = 
\begin{cases}
1_{C^*(G)}=\delta_e & \text{if } x=y \\
0 & \text{otherwise,}
\end{cases}
\]
and then $\{e_x : x\in X\}$ is an orthonormal basis for $M$ with reconstruction formula 
\begin{equation}\label{recon}
m = \sum_{x \in X} e_x \cdot \langle e_x,m\rangle\quad\text{for $m \in M$.}
\end{equation}
The left action of $C^*(G)$ on $M$ will be the integrated form of the unitary representation $T$ in the next proposition.

\begin{prop}\label{leftact}
Let $(G,X)$ be a self-similar action, and let $g\in G$. Then there is an adjointable operator $T_g$ on $M$ such that 
\begin{equation}\label{defleftaction}
T_g (e_x\cdot a)= e_{g\cdot x}\cdot (\delta_{g|_{x}}a)\quad\text{for $x\in X$ and $a\in C^*(G)$,}
\end{equation}
and $T:g\mapsto T_g$ is a unitary representation of $G$ in $\LL(M)$.
\end{prop}

\begin{proof}
We define $T_g:M \to M$ by
\[
T_g(m)= \sum_{y \in X} e_{g \cdot y} \cdot \big(\delta_{g|_y} \langle e_y,m \rangle\big).
\]
For $a\in C^*(G)$ and $m\in M$, we have
\[
T_g(m \cdot a)=\sum_{y \in X} e_{g \cdot y} \cdot (\delta_{g|_y} \langle e_y,m \cdot a \rangle)=\sum_{y \in X} e_{g \cdot y} \cdot (\delta_{g|_y} \langle e_y,m \rangle a) = T_g(m)\cdot a,
\]
and hence $T_g$ is $C^*(G)$ linear. The computation
\begin{align*}
T_g(e_x \cdot a) = \sum_{y \in X} e_{g \cdot y} \cdot \big(\delta_{g|_y} \langle e_y,e_x \cdot a \rangle\big)=e_{g \cdot x} \cdot (\delta_{g|_x} a)
\end{align*}
shows that $T_g$ satisfies Equation \eqref{defleftaction}.

We next show that $T_g$ is adjointable with $T_g^*=T_{g^{-1}}$. Let $g\in G$, $x,y\in X$ and $a,b\in C^*(G)$. Then
\begin{align*}
\langle T_g (e_x\cdot a) , e_y\cdot b \rangle 
& = 
\langle e_{g \cdot x}\cdot (\delta_{g|_{x}}a) , e_y\cdot b\rangle \\
&=
\begin{cases}
\left( \delta_{g|_{x}}a \right)^* b & \text{if } y=g \cdot x \\
0 & \text{otherwise}
\end{cases} \\
&=
\begin{cases}
a^* \delta_{(g|_{x})^{-1}}  b & \text{if } y=g \cdot x \\
0 & \text{otherwise}
\end{cases} \\
&=
\begin{cases}
a^* \delta_{g^{-1}|_{g \cdot x}} b & \text{if } y=g \cdot x \\
0 & \text{otherwise}
\end{cases} \\
&=
\begin{cases}
a^* \delta_{g^{-1}|_{y}} b & \text{if } x=g^{-1} \cdot y \\
0 & \text{otherwise}
\end{cases} \\
&= 
\langle  e_x\cdot a , e_{g^{-1}\cdot y}\cdot \delta_{g^{-1}|_{y}} b \rangle\\
& =\langle  e_x\cdot a , T_{g^{-1}} (e_y\cdot b) \rangle. 
\end{align*}
which implies that $T_g$ is adjointable with $T_g^*=T_{g^{-1}}$. Next we let $g,h\in G$, and the calculation
\begin{align*}\label{Tmult}
T_{gh}(e_x\cdot a) 
&= 
e_{(gh)\cdot x}\cdot (\delta_{gh|_{x}} a)\notag \\
&=
e_{g \cdot (h\cdot x)}\cdot (\delta_{g|_{h\cdot x} h|_x} a)
\quad\text{by Lemma~\ref{lem:props}} \notag\\
&=
T_g (e_{h\cdot x} \cdot (\delta_{h|_{x}}a))\notag\\
&=
T_g T_h (e_x\cdot a).
\end{align*}
shows that $T_{gh}=T_gT_h$. Since $T_g^*=T_{g^{-1}}$, this implies that each $T_g$ is unitary, and that $T$ is a homomorphism of $G$ into the unitary group $\UU\LL(M)$, or, in other words, a unitary representation in $\LL(M)$.
\end{proof}

By \cite[Proposition C.17]{tfb}, the unitary representation $T\colon G \to \UU\LL(M)$ has an integrated form $\pi_T\colon C^*(G)\to \LL(M)$ satisfying $\pi_T(\delta_g)=T_g$, and with the left action defined by $a\cdot m=\pi_T(a)m$, $M$ becomes a Hilbert bimodule over $C^*(G)$.

A representation of $M$ in a $C^*$-algebra $B$ consists of a linear map $\psi:M \to B$ and a homomorphism $\pi:C^*(G) \to B$ satisfying
\begin{align*}
\psi(x \cdot a) &= \psi(x)\pi(a), \\
\psi(x)^*\psi(y) &= \pi(\langle x,y\rangle_{C^*(G)}), \text{ and} \\
\psi(a \cdot x) &= \pi(a)\psi(x)
\end{align*}
for all $x,y \in X$ and $a \in C^*(G)$ (see~\cite[Section 1]{fr}). A representation $(\psi,\pi)$ of $M$ in $B$ induces a homomorphism $(\psi,\pi)^{(1)}:\KK(M)\to B$ such that $(\psi,\pi)^{(1)}(\Theta_{m,n})=\psi(m)\psi(n)^*$ for $m,n\in M$ \cite[Proposition~1.6]{fr}, and $(\psi,\pi)$ is \emph{Cuntz-Pimsner covariant} if\footnote{This is Pimsner's original definition \cite{p}; many authors use a slightly different definition due to Katsura, but the two definitions give the same algebras for the bimodules we consider.} 
\[
(\psi,\pi)^{(1)}(\pi_T(a))=\pi(a)\quad\text{whenever $\pi_T(a)\in \KK(M)$.}
\]

By \cite[Proposition 1.3]{fr}, the Hilbert bimodule $M$ has a Toeplitz algebra $\TT(M)$ generated by a universal representation $(i_M,i_{C^*(G)}):M \to \TT(M)$; if $(\psi,\pi)$ is a representation of $M$ in $B$, we write $\psi \times \pi$ for the homomorphism of $\TT(M)$ into $B$ such that $(\psi \times \pi) \circ i_M = \psi$ and $(\psi \times \pi) \circ i_{C^*(G)} = \pi$. The Cuntz-Pimsner algebra $\OO(M)$ is the quotient of $\TT(M)$ which is generated by a universal Cuntz-Pimsner covariant representation $(j_M,j_{C^*(G)})$. We call $\TT(G,X):=\TT(M)$ and $\OO(G,X):=\OO(M)$ the \emph{Toeplitz algebra} and \emph{Cuntz-Pimsner algebra} of the self-similar action $(G,X)$. It will follow from Corollary~\ref{CPquotient} below that $\OO(G,X)$ is the same as the universal Cuntz-Pimsner algebra $\OO_G$ in \cite[Definition~3.1]{nekra}.

We will use the following presentation of $\TT(G,X)$.

\begin{prop} \label{Toeplitz}
Let $(G,X)$ be a self-similar action, and set $u_g:=i_{C^*(G)}(\delta_g)$ for $g\in G$, and $s_x:=i_M(e_x)$ for $x\in X$. Then
\begin{enumerate}
\item\label{uunitary} $u:G\to \TT(G,X)$ is a unitary representation of $G$,
\smallskip

\item\label{sxTC} $\{s_x:x\in X\}$ is a Toeplitz-Cuntz family of isometries in $\TT(G,X)$, and
\smallskip

\item\label{ssTC} $u_g s_x = s_{g\cdot x} u_{g|_{x}}$ for $g\in G$ and $x\in X$.
\end{enumerate}
The set $\{u_g:g\in G\}\cup\{s_x:x\in X\}$ generates $\TT(G,X)$, and $(\TT(G,X),(u,s))$ is universal for families $\{U_g:g\in G\}$ and $\{S_x:x\in X\}$ satisfying \textnormal{\eqref{uunitary}, \eqref{sxTC}} and \textnormal{\eqref{ssTC}}.
\end{prop}

\begin{proof}
The map $u$ is a unitary representation because $\delta:G\to UC^*(G)$ is, and $i_{C^*(G)}$ is a unital homomorphism (which follows from \cite[Corollary~3.3]{br}). We have
\[
s_x^* s_y =  i_M(e_x)^*  i_M(e_y) = 
i_{C^*(G)} (\langle e_x, e_y \rangle) = 
\begin{cases}
1 & \text{if } x=y \\
0 & \text{if } x\neq y,
\end{cases}
\]
which implies that $\{s_x:x\in X\}$ is a Toeplitz-Cuntz family. For \eqref{ssTC}, we compute
\begin{align*}
u_g s_x &= i_{C^*(G)}(\delta_g) i_M (e_x)
= i_M (\delta_g \cdot e_x)=i_M(T_g(e_x))\\&= i_M (e_{g\cdot x}\cdot \delta_{g|_{x}})
= i_M (e_{g\cdot x} )  i_{C^*(G)}(\delta_{g|_{x}})
= s_{g\cdot x} u_{g|_{x}}.
\end{align*}
The $u_g$ generate $i_{C^*(G)}(C^*(G))$, and for $m\in M$, the reconstruction formula \eqref{recon} gives
\[
i_M(m)=i_M\Big(\sum_{x\in X}e_x\cdot \langle e_x,m\rangle\Big)=\sum_{x\in X}i_M(e_x)i_{C^*(G)}\big(\langle e_x,m\rangle\big).
\]
Thus $C^*(u_g,s_x)$ contains all the generators of $\TT(G,X)=\TT(M)$, and must be all of $\TT(G,X)$.

To see the  universal property, suppose $D$ is a $C^*$-algebra, and
$\{U_g\}\subset D$ and $\{S_x\}\subset D$ satisfy \eqref{uunitary}, \eqref{sxTC} and \eqref{ssTC}. We have to find a homomorphism $\pi_{U,S}:\TT(G,X)\to D$ such that $\pi_{U,S}(u_g)=U_g$ and $\pi_{U,S}(s_x)=S_x$.
Let $\pi_U:C^*(G)\to D$ be the integrated form of $U$. Since each element of $M$ has a unique expansion $\sum_{x\in X}e_x\cdot a_x$, there is a well-defined linear function $\psi:M\to M$ such that $\psi(e_x\cdot a)=S_x\pi_U(a)$ for $x\in X$ and $a\in C^*(G)$. 

We claim that $(\psi,\pi_U)$ is a representation of $M$. Let $a\in C^*(G)$ and $x\in X$. Then
\[
\psi((e_x\cdot a)\cdot b)=\psi(e_x\cdot(ab))=S_x\pi_U(ab)=(S_x\pi_U(a))\pi_U(b)=\psi(e_x\cdot a)\pi_U(a).
\]
Next we consider the left action of $b=\delta_g$, which is implemented by the operator $T_g$ of Proposition~\ref{leftact}. We calculate using relation \eqref{ssTC}:
\begin{align*}
\psi(\delta_g\cdot(e_x\cdot a)) &= 
\psi( e_{g\cdot x}\cdot (\delta_{g|_{x}} a) ) 
= S_{g\cdot x} \pi_U(\delta_{g|_{x}} a) \\
&= S_{g\cdot x} U_{g|_{x}} \pi_U(a)
= U_g S_x \pi(a) 
=\pi_U(\delta_g)\psi(e_x\cdot a),
\end{align*}
which implies that $\psi(a\cdot m) = \pi(a)\psi(m)$ for $a\in C^*(G)$ and $m\in M$.
For the inner product, we have 
\begin{align*}
\pi_U(\langle e_x\cdot a , e_y\cdot b \rangle ) &= 
\begin{cases}
\pi_U(a^*b) & \text{if } x=y \\
0 & \text{if } x\neq y
\end{cases}
\\
&=
\pi_U(a)^* S_x^* S_y \pi_U(b) =(S_x\pi_U(a))^* (S_y\pi_U(b))\\ 
&= \psi(e_x\cdot a)^*\psi(e_y\cdot b ).
\end{align*}
So $(\psi,\pi_U)$ is a Toeplitz representation of $M$, as claimed. Thus there is a homomorphism $\psi\times \pi_U:\TT(G,X)=\TT(M)\to D$. We have 
\[
(\psi\times\pi_U)(u_g)=(\psi\times\pi_U)(i_{C^*(G)}(\delta_g))=\pi_U(\delta_g)=U_g,
\]
and similarly $(\psi\times\pi_U)(s_x)=S_x$. Thus $\pi_{U,S}:=\psi\times\pi_U$ has the required properties.
\end{proof}

We now recall some standard notation for working with the Toeplitz-Cuntz family $\{s_x:x\in X\}$. For $v\in X^n$, we write $s_v:=s_{v_1}s_{v_2}\cdots s_{v_n}$. Then for each $n$, $\{s_v:v\in X^n\}$ is a Toeplitz-Cuntz family, so we have $1\geq \sum_{v\in X^n}s_vs_v^*$. For $v,w\in X^*$, the product $s_v^*s_w$ vanishes unless either $v=wv'$ or $w=vw'$, and then collapses down to $s_{v'}^*$ or $s_{w'}$. The relation \eqref{ssTC} in Proposition~\ref{Toeplitz} extends to $u_gs_v=s_{g\cdot v}u_{g|_v}$ for $v\in X^*$.

\begin{cor}\label{spanning}
Let $(G,X)$ be a self-similar action, and take $(u,s)$ as in Proposition~\ref{Toeplitz}. Then
\[
\TT(G,X)=\clsp\{s_v u_g s_w^*:v,w \in X^*,\; g \in G \}.
\]
\end{cor}

As usual, we prove that $A_0:=\newspan\{s_v u_g s_w^*\}$ is a $*$-subalgebra of $\TT(G,X)$, and then since $A_0$ contains all the generators $\{u_g\}\cup \{s_x\}$, its closure has to be all of $\TT(G,X)$. Since $A_0$ is closed under taking adjoints, it remains to show that $\{s_v u_g s_w^*\}$ is closed under multiplication. Since we  will need the result of the computation, we state it separately:

\begin{lemma}\label{lem:spanning}
For $v,w,y,z\in X^*$ and $g,h\in G$, we have
\begin{equation}
\label{product}
(s_v u_g s_w^*)(s_{y} u_{h} s_{z}^*)=
\begin{cases}
s_{v(g \cdot y')} u_{g|_{y'} h} s_{z}^* & \textrm{ if } y=wy' \\
s_v u_{g(h|_{h^{-1} \cdot w'})} s_{z(h^{-1} \cdot w')}^* & \textrm{ if } w=yw' \\
0 & \textrm{ otherwise. }
\end{cases}
\end{equation}
\end{lemma}

\begin{proof}
We have $s_w^* s_y = 0$ unless either $y=wy'$ or $w=yw'$, and hence a computation using the relations $u_g s_w = s_{g \cdot w} u_{g|_{w}}$ 
gives
\begin{align*}
s_v u_g s_w^* s_{y} u_{h} s_{z}^* &=
\begin{cases}
s_v u_g s_{y'} u_{h} s_{z}^* & \textrm{ if } y=wy' \\
s_v u_g s_{w'}^*u_{h} s_{z}^* & \textrm{ if } w=yw' \\
0 & \textrm{ otherwise }
\end{cases} \\
&=
\begin{cases}
s_{v}s_{g \cdot y'} u_{g|_{y'}}u_{h} s_{z}^* & \textrm{ if } y=wy' \\
s_v u_g (s_{h^{-1} \cdot w'} u_{{h^{-1}}|_{w'}})^*  s_{z}^* & \textrm{ if } w=yw' \\
0 & \textrm{ otherwise. }
\end{cases} \\
&=
\begin{cases}
s_{v(g \cdot y')} u_{g|_{y'} h} s_{z}^* & \textrm{ if } y=wy' \\
s_v u_{g (h|_{h^{-1} \cdot w'})} s_{z(h^{-1} \cdot w')}^* & \textrm{ if } w=yw' \\
0 & \textrm{ otherwise,}
\end{cases}
\end{align*}
as required.
\end{proof}

\begin{cor}\label{CPquotient}
Let $(G,X)$ be a self-similar action, and take $(u,s)$ as in Proposition~\ref{Toeplitz}. Then $\OO(G,X)$ is the quotient of $\TT(G,X)$ by the ideal $I$ generated by $1-\sum_{x\in X}s_xs_x^*$.
\end{cor}

\begin{proof}
Since $\{e_x:x\in X\}$ is an orthonormal basis for $M$, it follows from \cite[Lemma~2.5]{ehr} that a Toeplitz representation $(\psi,\pi)$ is Cuntz-Pimsner covariant if and only if
\[
1=\sum_{x\in X}\psi(e_x)\psi(e_x)^*=\sum_{x\in X}(\psi\times\pi)(i_M(e_x)i_M(e_x)^*)=(\psi\times\pi)\Big(\sum_{x\in X}s_xs_x^*\Big),
\]
and hence if and only if $\psi\times \pi$ vanishes on $I$. 
\end{proof}

We write $u_g$ and $s_x$ also for the images of the generators in $\OO(G,X)$. Since $1=\sum_{x\in X}s_xs_x^*$ in $\OO(G,X)$, the image of the Toeplitz-Cuntz family $\{s_x:x\in X\}$ in $\OO(G,X)$ is a Cuntz family. The same is true of the Toeplitz-Cuntz families $\{s_v:v\in X^n\}$, so for every $n\in\N$ we have $1=\sum_{v\in X^n}s_vs_v^*$ in $\OO(G,X)$.

\begin{remark} \label{remark:cfNek}
As we observed earlier, Corollary~\ref{CPquotient} implies that $\OO(G,X)$ is the universal Cuntz-Pimsner algebra $\OO_G$ in \cite[Definition~3.1]{nekra}. It is not necessarily the same as the Cuntz-Pimsner algebra in \cite{nek_jot}, which is generated by a Cuntz family $\{s_x:x\in X\}$ and a unitary representation $u$ of $G$ which factors through a particular ``permutation representation'' of $C^*(G)$. 
\end{remark}

\begin{cor}\label{Cuntz-finite}
Let $(G,X)$ be a self-similar action with nucleus $\NN$. Then 
\begin{equation}\label{smallspanset}
\OO(G,X)=\clsp\{s_v u_g s_w^* : v,w \in X^*, g \in \NN \}.
\end{equation}
\end{cor}

\begin{proof}
Since $\OO(G,X)$ is a quotient of $\TT(G,X)$, Corollary \ref{spanning} implies that the elements $\{s_v u_h s_w^* : v,w \in X^*, h \in G\}$ span a dense subspace of $\OO(G,X)$. We will show that each $s_v u_h s_w^*$ belongs to the right-hand side of \eqref{smallspanset}. Since $(G,X)$ has nucleus $\NN$, there exists $n \in \N$ such that $h|_y \in \NN$ for all $y \in X^n$. But then the Cuntz relation $1= \sum_{y \in X^n} s_y s_y^*$ gives
\begin{align*}
s_vu_hs_w^*&= s_vu_h \Big(\sum_{y \in X^n} s_y s_y^*\Big)s_w^* = \sum_{y \in X^n} s_{v(h \cdot y)} u_{h|_{y}} s_{wy}^*,
\end{align*}
which belongs to the right-hand side of \eqref{smallspanset}.
\end{proof}

\begin{cor}
If $(G,X)$ is contracting with trivial nucleus $\NN = \{e\}$, then  $\OO(G,X)$ is the Cuntz algebra $\OO_{|X|}$.
\end{cor}

\begin{proof}
If $\NN = \{e\}$, then Corollary \ref{Cuntz-finite} implies that $\OO(G,X)$ is generated by the Cuntz family $\{s_x \mid x \in X \}$, and hence by the uniqueness theorem for the Cuntz algebra is canonically isomorphic to $\OO_{|X|}$.
\end{proof}

\subsection{ Universal algebras associated with integer matrices}\label{dilation_Toeplitz}

We consider a matrix $A\in M_d(\Z)$ with $N=|\det A|>1$, and  the associated self-similar group $(\Z^d,\Sigma)$ of \S\ref{dilation_example}. We want to show that $\TT(\Z^d,\Sigma)$ and $\OO(\Z^d,\Sigma)$ are the Toeplitz algebra $\TT(C(\T^d),\alpha_A,L)$ and Exel crossed product $C(\T^d)\rtimes_{\alpha_A,L}\N$ studied in \cite{lrr}. 

As in \cite{lrr} and \cite{ehr}, we consider the $N$-to-$1$ covering map $\sigma_A: \T^d \to \T^d$ such that $\sigma_A(e^{2\pi ix})=e^{2\pi iAx}$, and the endomorphism $\alpha_A: f \to f \circ \sigma_A$ of $C(\T^d)$. The function $L:C(\T^d)\to C(\T^d)$ defined by
\[ 
L(f)(z) = \frac{1}{N} \sum_{\sigma_A(w) = z} f(w). 
\]
is a  transfer operator for $\alpha_A$, and $(C(\T^d),\alpha_A,L)$ is the Exel system studied in \cite{ehr} and \cite{lrr}. Following \cite{br}, we write $M_L$ for the associated Hilbert bimodule over $C(\T^d)$, with inner product and operations given by
\[
\langle m,n\rangle:=L(m^*n),\quad 
f \cdot m:=fm\quad \text{and}\quad m \cdot f:=m\alpha_A(f)
\]
for $m,n \in M_L$ and $f \in C(\T^d)$. It is shown in \cite[Lemma~3.3]{LR2} that $M_L$ is complete in the norm defined by the inner product. 
The Toeplitz algebra $\TT(C(\T^d),\alpha_A,L)$ in \cite{lrr} is by definition the Toeplitz algebra $\TT(M_L)$, and the Exel crossed product $C(\T^d)\rtimes_{\alpha_A,L}\N$ is the quotient $\OO(M_L)$ (by \cite{br}).

In \cite[Proposition~3.1]{lrr}, we showed that $\TT(M_L)$ is the universal algebra generated by a unitary representation $u:\Z^d\to U\TT(M_L)$ and an isometry $v$ satisfying 
\begin{itemize}
\item[(E1)] $v u_m = u_{Bm} v$,
\item[(E2)] $ v^* u_m v = \begin{cases}
u_{B^{-1}m} & \text{ if } m \in B \Z^d \\
0 & \text{ otherwise}.  \\
\end{cases}$
\end{itemize}
We will use this presentation and that of Proposition~\ref{Toeplitz} to identify $\TT(\Z^d,\Sigma)$ with $\TT(M_L)$. (The use of the same letter $u$ for the unitary representation of $\Z^d$ in both presentations should not cause problems because our isomorphism takes one $u_n$ to the other $u_n$.)

\begin{prop}\label{isoToeplitz}
Suppose that $A\in M_d(\Z)$ has $|\det A|>1$, write $B:=A^t$, and consider the bimodule $M_L$ constructed above. Define $b:\Sigma^*\to \Z^d$ by $b(w)=w_1+Bw_2+\cdots +B^{k-1}w_k$ for $w\in \Sigma^k$. Then there is an isomorphism $\theta$ of $\TT(\Z^d,\Sigma)=C^*(u,s)$ onto $\TT(M_L)=C^*(u,v)$ such that
\begin{align}\label{thetagen1}
\theta(s_yu_ns_w^*)&=u_{b(y)+B^kn}v^kv^{*l}u_{b(w)}^*\\
&=u_{b(y)}v^kv^{*l}u_{b(w)+B^ln}^*\label{thetagen2}
\end{align}
for $y\in \Sigma^k$, $w\in \Sigma^l$, $n\in\Z^d$.
\end{prop}

\begin{proof}
We begin by building a representation of $\TT(\Z^d,\Sigma)$ in $\TT(M_L)=C^*(u,v)$. We have the $u_n$, and they satisfy condition \eqref{uunitary} of Proposition~\ref{Toeplitz} because $u:\Z^d\to \TT(M_L)$ is a unitary representation. For $x\in \Sigma$, we define $S_{x} = u_{x}v$. Then for $x,y\in \Sigma$, property (E2) gives
\[
S_{x}^*S_{y}=v^* u_{-x}u_{y}v=v^* u_{y-x}v
=\begin{cases}
u_{B^{-1}(y-x)} & \text{ if } y-x\in B\Z^d\\
0 & \text{ otherwise;} \\
\end{cases}
\]
since both $x$ and $y$ are in $\Sigma$, $y-x\in B\Z^d$ if and only if $x=y$. Thus the $\{S_x\}$ are isometries with orthogonal ranges, and form a Toeplitz-Cuntz family, as required in Proposition~\ref{Toeplitz}\,\eqref{sxTC}. Next we use (E1):
\begin{align*}
u_nS_x&=u_{n+x}v=u_{c(n+x)}u_{n+x-c(n+x)}v\\
&=u_{c(n+x)}vu_{B^{-1}(n+x-c(n+x))}\\
&=S_{c(n+x)}u_{B^{-1}(n+x-c(n+x))},
\end{align*}
which is $S_{n\cdot x}u_{n|_x}$ by \eqref{idrest}. Now the universal property of $(\TT(\Z^d,\Sigma),u,s)$ gives us a homomorphism $\theta=\theta_{u,S}:\TT(\Z^d,\Sigma)\to \TT(M_L)$ such that $\theta(s_x)=u_xv$ and $\theta\circ u=u$. The range contains all the generators $u_n$ and $v=S_0$, and hence $\theta$ is onto. 

To see that $\theta$ is injective, we build an inverse. We define $V:=s_0$. Then $V$ is certainly an isometry. For $m\in \Z^d$, an application of \eqref{idrest} gives 
\[
u_{Bm}V=u_{Bm}s_0=s_{c(0+Bm)}u_{Bm|_0}=s_0u_{B^{-1}(Bm+0-c(Bm+0))}=Vu_m,
\]
which is (E1). For (E2), we use \eqref{idrest} again:
\[
V^*u_mV=s_0^*u_ms_0=s_0^*s_{c(m)}u_{B^{-1}(m-c(m))};
\]
since the $s_x$ have mutually orthogonal ranges, this last term vanishes unless $c(m)=0$, or equivalently $m\in B\Z^d$, in which case it is $s_0^*s_{0}u_{B^{-1}m}=u_{B^{-1}m}$. Now the universal property of $\TT(M_L)$ gives a homomorphism $\theta':\TT(M_L)\to \TT(\Z^d,\Sigma)$ such that $\theta'(u_m)=u_m$ and $\theta'(v)=V$. Then $\theta'\circ \theta(s_x)=\theta'(u_xv)=u_xs_0=s_x$, so $\theta'\circ \theta$ is the identity, and $\theta$ is injective.

To check the formulas for $\theta$ on spanning elements, we take $y\in \Sigma^k$, $w\in \Sigma^l$ and $n\in\Z^d$. Then
\begin{align*}
\theta(s_yu_ns_w^*)&=\theta(s_{y_1}s_{y_2}\cdots s_{y_k}u_ns_w^*)=u_{y_1}vu_{y_2}v\cdots u_{y_k}vu_ns_w^*\\
&=u_{y_1}u_{By_2}v^2\cdots u_{y_k}vu_ns_w^*\\
&=u_{y_1+By_2+\cdots +B^{k-1}y_k}v^ku_nv^{*l}u^*_{w_1+Bw_2+\cdots+B^{l-1}w_l}\\
&=u_{y_1+By_2+\cdots +B^{k-1}y_k+B^kn}v^kv^{*l}u^*_{w_1+Bw_2+\cdots+B^{l-1}w_l},
\end{align*}
which is the first formula \eqref{thetagen1}. For \eqref{thetagen1}, instead of pulling $u_n$ past $v^k$ using (E1) at the last step, pull it past $v^{*l}$ using the adjoint of (E1).
\end{proof}

\begin{cor}
Suppose that $A$ and $M_L$ are as above. Then the isomorphism $\theta$ of Proposition~\ref{isoToeplitz} induces an isomorphism $\bar\theta$ of $\OO(\Z^d,\Sigma)=C^*(u_n,s_x)$ onto $\OO(M_L)=C^*(\bar u_n,\bar v)$ such that
\[
\bar\theta(s_v u_n s_w^*)=\bar u_{b(v)+B^kn}\bar v^k\bar v^{*l}\bar u_{b(w)}^*.
\]
\end{cor}

\begin{proof}
We know from Corollary~\ref{CPquotient} that $\OO(\Z^d,\Sigma)$ is the quotient of $\TT(\Z^d,\Sigma)$ by the ideal $I$ generated by $1-\sum_{x\in \Sigma}s_xs_x^*$, and from \cite[Proposition~3.3]{lrr} that $\OO(M_L)$ is the quotient of $\TT(M_L)$ by the ideal $J$ generated by $1-\sum_{x\in \Sigma}(u_xv)(u_xv)^*$. Since $\theta(s_x)=u_{b(x)}v=u_xv$, we have 
\[
\theta\Big(1-\sum_{x\in \Sigma}s_xs_x^*\Big)=1-\sum_{x\in \Sigma}(u_xv)(u_xv)^*,
\]
so $\theta(I)=J$, and the result follows.
\end{proof}

\section{A characterisation of KMS states}\label{Sec:char}

Let $(G,X)$ be a self-similar action. The Toeplitz algebra $\TT(G,X)=\TT(M)$ carries a strongly continuous gauge action $\gamma: \T \to \Aut(\TT(G,X))$  such that $\gamma_z(i_{C^*(G)}(a))=i_{C^*(G)}(a)$ for $a \in C^*(G)$ and $\gamma_z(i_M(m))=z i_M(m)$ for $m \in M$. We define $\sigma:\R \to \Aut(\TT(G,X))$ by $\sigma_t=\gamma_{e^{it}}$. In terms of the presentation of Proposition~\ref{Toeplitz}, we have
\[ 
\sigma_t(u_g)=u_g\quad\text{and}\quad\sigma_t(s_v)=e^{it|v|}s_v. 
\]
We also write $\sigma$ for the induced action of $\R$ on $\OO(G,X)$. Our main goal is to find the KMS states of $(\TT(G,X),\sigma)$ and  $(\OO(G,X),\sigma)$. In this section, we give a characterisation of KMS states which will make them easier to identify.

Our conventions for KMS states are the same as those of \cite{lr} and \cite{lrr}, and are explained at the beginning of \cite[\S7]{lr}, for example. For our purposes, a state $\phi$ of a system $(B,\R,\alpha)$ is a \emph{KMS state with inverse temperature $\beta$} (a KMS$_\beta$ state) if $\phi(ab)=\phi(b\alpha_{i\beta}(a))$ for all $a,b$ in a family $\FF$ of analytic elements which span a dense subspace of $B$. We distinguish between KMS$_\infty$ states, which are by definition limits of KMS$_\beta$ states as $\beta\to \infty$, and ground states, for which $z\mapsto \phi(a\alpha_z(b))$ is bounded in the upper-half plane for all $a,b\in \FF$. (This distinction is not made in the standard references \cite{bra-rob,ped}.)

The spanning elements $s_v u_g s_w^* \in \TT(G,X)$ are analytic for $\sigma$ since
\begin{equation}\label{action}
\sigma_t(s_v u_g s_w^*) = e^{it(|v|-|w|)} s_v u_g s _w^* 
\end{equation}
and the function $z \mapsto e^{iz(|v|-|w|)}$ is entire. Thus a state $\phi$ is KMS$_\beta$ for $\sigma$ if and only if
\begin{align}
\label{KMS_iff}
\phi\big((s_v u_g s_w^*)(s_y u_h s_z^*)\big)&=\phi\big((s_y u_h s_z^*)\sigma_{i\beta}(s_v u_g s_w^*)\big)\notag\\
&=e^{-\beta(|v|-|w|)} \phi\big((s_y u_h s_z^*)(s_v u_g s_w^*)\big).
\end{align}
We now have the following analogue of \cite[Lemma 8.3]{lr} and \cite[Proposition 4.1]{lrr}.

\begin{prop}\label{KMS>beta}
Let $(G,X)$ be a self-similar action and suppose that $\sigma:\R \rightarrow \Aut\TT(G,X)$ satisfies~\eqref{action}. 
\begin{enumerate}
\item\label{ida} For $\beta < \log \d$, there are no KMS$_\beta$-states for $\sigma$.
\item\label{idb} For $\beta \geq \log \d$, a state $\phi$ is a KMS$_\beta$-state for $\sigma$ if and only if
\begin{equation}\label{phi_trace}
\phi(u_g u_h)=\phi(u_h u_g)\quad\text{for $g,h \in G$}
\end{equation}
and 
\begin{equation}\label{char:eqn}
\phi(s_v u_g s_w^*) =
\begin{cases}
e^{-\beta|v|} \phi(u_g) & \textrm{ if } v=w \\
0 & \textrm{ otherwise}
\end{cases} \\
\end{equation}

\end{enumerate}
\end{prop}

\begin{proof}
Suppose that $\phi$ is a KMS$_\beta$-state. First, for $g,h \in G$ we have
\[ \phi(u_g u_h)=\phi(u_h \sigma_{i\beta}(u_g)) = \phi(u_h u_g). \]
Next, we take $v,w \in X^*$ and calculate
\begin{align*}
\phi(s_v u_g s_w^*) &= \phi(u_g s_w^*\sigma_{i\beta}(s_v)) \\
&=
\begin{cases}
e^{-\beta|v|}\phi(u_g s_w^* s_v) & \textrm{ if } v=wv' \textrm{ or } w=vw' \\
0 & \textrm{ otherwise }
\end{cases} \\
&= 
\begin{cases}
e^{-\beta|v|}\phi(s_v \sigma_{i\beta}(u_g s_w^*)) & \textrm{ if } v=wv' \textrm{ or } w=vw' \\
0 & \textrm{ otherwise }
\end{cases} \\
&=
\begin{cases}
e^{-\beta(|v|-|w|)}\phi(s_v u_g s_w^*) & \textrm{ if } v=wv' \textrm{ or } w=vw' \\
0 & \textrm{ otherwise. }
\end{cases}
\end{align*}
Thus 
\[
\phi(s_v u_g s_w^*) \neq 0 \implies |v|=|w| \text{ and either } v=wv' \text{ or } w=vw' \iff  v=w. \]
If $v=w$, then
\[ 
\phi(s_v u_g s_v^*) = \phi(u_g s_v^* \sigma_{i\beta}(s_v)) = e^{-\beta |v|}\phi(u_g s_v^* s_v)= 
e^{-\beta |v|}\phi(u_g).
\]
Since $\{s_x:x \in X\}$ is a Toeplitz-Cuntz family, we have
\[ 
1 = \phi(1) \geq \phi \Big( \sum_{x \in X} s_x s_x^*\Big) 
= \sum_{x \in X} \phi(s_x s_x^*) = \sum_{x \in X} e^{-\beta}\phi(s_x^* s_x) = \sum_{x \in X} e^{-\beta} = \d e^{-\beta}, 
\]
so that $\beta \geq \log \d$. This completes the proof of \eqref{ida} and the forward implication in \eqref{idb}.

For the backward implication in \eqref{idb}, suppose $\phi$ is a tracial state on $C^*(G)$ satisfying \eqref{phi_trace} and \eqref{char:eqn}. We aim to show that $\phi$ satisfies \eqref{KMS_iff}. We first suppose that $|y|\geq |w|$. Since $y$ is longer than $w$, the product $s_w^*s_y$ in the middle of the left-hand side of \eqref{KMS_iff} vanishes unless $y=wy'$. If $y=wy'$, then Lemma~\ref{lem:spanning} implies that 
\[
\phi \big((s_v u_g s_w^*)(s_{y} u_{h} s_{z}^*)\big)=\phi(s_{v(g \cdot y')} u_{g|_{y'} h} s_{z}^*),
\]
which by \eqref{char:eqn} vanishes unless $z=v(g \cdot y')$. For $y=wy'$ and $z=v(g \cdot y')$, we compute
\begin{align}
\phi\big((s_v u_g s_w^*)(s_{y} u_{h} s_{z}^*)\big)&=\phi(s_{v(g \cdot y')} u_{g|_{y'} h} s_{z}^*) \notag\\
&=e^{-\beta |v(g \cdot y')|} \phi(u_{g|_{y'}} u_{h}) \notag\\
&=e^{-\beta (|v|+|y'|)} \phi(u_{h} u_{g|_{y'}}) \quad\text{(using \eqref{phi_trace})}\notag\\
&=e^{-\beta (|v|+|y'|)}e^{\beta|y|} \phi(s_y u_{h} u_{g|_{y'}}s_y^*) \quad\text{(using \eqref{char:eqn})} \notag\\
&=e^{-\beta (|v|+|y'|-|y|)} \phi(s_y u_{h} u_{g|_{y'}}s_{y'}^*s_{w}^*)\notag\\
&=e^{-\beta (|v|-|w|)} \phi(s_y u_{h} (s_{y'} u_{g^{-1}|_{g \cdot y'}})^*s_{w}^*)\notag\\
&=e^{-\beta (|v|-|w|)} \phi(s_y u_{h} (s_{g^{-1}(g \cdot y')} u_{g^{-1}|_{g \cdot y'}})^*s_{w}^*)\notag\\
&=e^{-\beta (|v|-|w|)} \phi(s_y u_{h} (u_{g^{-1}} s_{g\cdot y'})^*s_{w}^*) \notag\\
&=e^{-\beta (|v|-|w|)} \phi(s_y u_{h} s_{g\cdot y'}^* u_{g}s_{w}^*)\notag\\
&=e^{-\beta (|v|-|w|)} \phi(s_y u_{h} s_{v(g\cdot y')}^*s_vu_{g}s_{w}^*)\notag\\
&=e^{-\beta (|v|-|w|)} \phi\big((s_y u_{h} s_{z}^*)(s_v u_{g}s_{w}^*)\big).\label{compLHS}
\end{align}
To complete the proof of \eqref{idb}, we observe that
\begin{align*}
\phi\big((s_y u_{h} s_{z}^*)&(s_v u_{g}s_{w}^*)\big)\not=0\\
&\Longrightarrow\text{either $v=zv'$ and $w=y(h\cdot v')$, or $z=vz'$ and $y=w(g^{-1}\cdot z')$}\\
&\Longrightarrow \text{$z=vz'$ and $y=w(g^{-1}\cdot z')$}\quad\text{(because $|y|\geq |w|$)}\\
&\Longrightarrow \text{$z=v(g\cdot y')$ and $y=wy'$}\quad\text{(with $y'=g^{-1}\cdot z'$).}
\end{align*}
Thus if there is no $y'$ satisfying $y=wy'$ and $z=v(g\cdot y')$, we have
\begin{equation}\label{whenvanishes}
e^{-\beta (|v|-|w|)} \phi\big((s_y u_{h} s_{z}^*)(s_v u_{g}s_{w}^*)\big)=0=\phi\big((s_v u_g s_w^*)(s_y u_h s_z^*)\big).
\end{equation}
Together, \eqref{compLHS}, which holds when there exists $y'$ such that $y=wy'$ and $z=v(g \cdot y')$, and \eqref{whenvanishes}, which holds otherwise, imply \eqref{KMS_iff} for $|y|\geq |w|$.

We now suppose that $|y|<|w|$, and take adjoints to reduce to the case in the previous paragraph. Since $\phi(a)=\overline{\phi(a^* )}$, we have
\begin{align}
\phi\big((s_v u_g s_w^*)(s_y u_h s_z^*)\big)&=\overline{\phi\big((s_z u_{h^{-1}}s_y^*)(s_w u_{g^{-1}} s_v^*)\big)}\notag\\
&=e^{-\beta(|z|-|y|)}\overline{\phi\big((s_w u_{g^{-1}} s_v^*)(s_z u_{h^{-1}}s_y^*)\big)}\notag\\
&=e^{-\beta(|z|-|y|)} \phi\big((s_y u_h s_z^*)(s_v u_g s_w^*)\big).\label{takeadjoint}
\end{align}
The calculation in the previous paragraph shows that the right-hand side of \eqref{takeadjoint} vanishes unless $w=yw'$ and $v=z(g^{-1}\cdot w')$, in which case $|v|-|z|=|w'|=|w|-|y|$ and $|z|-|y|=|v|-|w|$. Thus
\[
\phi\big((s_v u_g s_w^*)(s_y u_h s_z^*)\big)
=e^{-\beta(|v|-|w|)} \phi\big((s_y u_h s_z^*)(s_v u_g s_w^*)\big),
\]
and we have proved \eqref{KMS_iff} in the remaining case $|y|<|w|$.
\end{proof}

\section{Existence of KMS states above the critical inverse temperature}\label{existence}

\begin{thm}\label{repn_hilbert}
Let $(G,X)$ be a self-similar action, and let $\tau$ be a normalised trace on $C^*(G)$. Then for every $\beta > \log|X|$, there is a KMS$_\beta$ state $\psi_{\beta,\tau}$ satisfying
\begin{equation} \label{repn_state}
\psi_{\beta,\tau}(s_v u_g s_w^*)=\begin{cases}
(1-|X|e^{-\beta}) \displaystyle{\sum_{k=0}^\infty e^{-\beta(k+|v|)} \Big(\sum_{\{y \in X^{k}\,:\,g \cdot y=y\}}}  \tau(\delta_{g|_{y}})\Big)& \text{ if } v=w \\
0 & \text{ otherwise.} \\
\end{cases}
\end{equation}
\end{thm}

To prove Theorem~\ref{repn_hilbert}, we adapt ideas from the proofs of \cite[Theorem~2.1]{ln} and \cite[Proposition~6.1]{lrr}. Both involve induced representations; as in \cite{ln} rather than \cite{lrr}, we apply Rieffel's Hilbert-bimodule formulation of induced representations to the Fock bimodule $\FF(M):=\bigoplus_{j=0}^\infty M^{\otimes j}$. We take $\pi_\tau:C^*(G)\to B(\KK_\tau)$ to be the GNS-representation of $C^*(G)$, and then Rieffel induction gives a representation
\begin{equation}\label{defpi}
\pi := \bigoplus_{j=0}^\infty M^{\otimes j}\dashind \pi_\tau \quad \text{on the Hilbert space} \quad  \HH_\pi:= \bigoplus_{j=0}^\infty M^{\otimes j} \otimes_{C^*(G)} \KK_\tau.
\end{equation}
To calculate with the induced representations $M^{\otimes j}\dashind \pi_\tau$, we need to understand the bimodules $M^{\otimes j}$.

\begin{lemma}\label{tenspowers}
Suppose that $(G,X)$ is a self-similar action and $\{e_x:x\in X\}$ is the orthonormal basis for the Hilbert bimodule $M$ constructed in \S\ref{sec:Toeplitz}. Fix $j\geq 1$. Then 
\[
\{e_v:=e_{v_1}\otimes\cdots\otimes e_{v_j}:v\in X^j\}
\]
is an orthonormal basis for $M^{\otimes j}$ with reconstruction formula
\begin{equation}\label{reconj}
m=\sum_{v\in X^j}e_v\cdot\langle e_v,m\rangle\quad\text{for $m\in M^{\otimes j}$.}
\end{equation}
The left action of $C^*(G)$ on $M^{\otimes j}$ satisfies $\delta_g\cdot e_v=e_{g\cdot v}\cdot \delta_{g|_{v}}$. 
\end{lemma}

\begin{proof}
We prove this result by induction on $j$. Proposition~\ref{leftact} and the surrounding discussion give the result for $j=1$. Suppose it is true for $j=k$. For two words $w=w_1w'$ and $v=v_1v'$ in $X^{k+1}$, we have $|w'|=|v'|=k$, and
\begin{align*}
\langle e_{w},e_{v}\rangle&=\langle e_{w_1}\otimes e_{w'},e_{v_1}\otimes e_{v'}\rangle=\langle e_{w'},\langle e_{w_1},e_{v_1}\rangle\cdot e_{v'}\rangle\\
&=\delta_{w_1,v_1}\langle e_{w'},e_{v'}\rangle=\delta_{w_1,v_1}\delta_{w',v'}1_{C^*(G)}=\delta_{w,v}1_{C^*(G)},
\end{align*}
giving orthonormality. For $m=m_1\otimes m'\in M^{\otimes (k+1)}=M\otimes_{C^*(G)} M^{\otimes k}$, we have
\begin{align*}
m&=\Big(\sum_{x\in X}e_x\cdot\langle e_x,m_1\rangle\Big)\otimes m'=\sum_{x\in X}e_x\otimes \langle e_x,m_1\rangle\cdot m'\\
&=\sum_{x\in X}e_x\otimes \Big(\sum_{x\in X^k}e_v\cdot\big\langle e_v,\langle e_x,m_1\rangle\cdot m'\big\rangle\Big)\\
&=\sum_{x\in X,\;v\in X^k}e_x\otimes e_v\cdot\langle e_x\otimes e_v,m_1\otimes m'\rangle,\\
&=\sum_{x\in X,\;v\in X^k}e_{xv}\cdot\langle e_{xv},m\rangle,
\end{align*}
which is the right-hand side of \eqref{reconj} for $j=k+1$. This formula extends by linearity and continuity of the inner product to $m\in M^{\otimes(k+1)}$. Finally, let $w=w_1w'\in X^{k+1}$ and $g\in G$. Then because the tensor product is balanced over $C^*(G)$, the inductive hypothesis gives
\begin{align*}
\delta_g\cdot e_w&=(\delta_g\cdot e_{w_1})\otimes e_{w'}=(e_{g\cdot w_1}\cdot \delta_{g|_{w_1}})\otimes e_{w'}\\
&= e_{g\cdot w_1}\otimes (\delta_{g|_{w_1}}\cdot e_{w'})=e_{g\cdot w_1}\otimes (e_{g|_{w_1}\cdot w'}\cdot\delta_{(g|_{w_1})|_{w'}})\\
&=(e_{g\cdot w_1}\otimes e_{g|_{w_1}\cdot w'})\cdot\delta_{g|_{w}}=e_{g\cdot w}\cdot\delta_{g|_{w}},
\end{align*}
and we now have the whole inductive hypothesis for $j=k+1$.
\end{proof}

\begin{proof}[Proof of Theorem~\ref{repn_hilbert}]
As promised, we take the GNS representation $\pi_\tau$ of $C^*(G)$ on $\KK_\tau$, and consider the representation $\pi$ of \eqref{defpi}. Lemma~\ref{tenspowers} implies that every vector in $M^{\otimes j} \otimes_{C^*(G)} \KK_\tau$ is a finite sum $\sum_{v\in X^j}e_v\otimes k_v$, and that the representation $M^{\otimes j}\dashind \pi_\tau$ of $C^*(G)$ is the integrated form of the unitary representation $U^j$ of $G$ on $M^{\otimes j} \otimes_{C^*(G)} \KK_\tau$ characterised by 
\begin{equation}\label{charUj}
U_g^j(e_v\otimes k)=T_g(e_v)\otimes k=(e_{g\cdot v}\cdot \delta_{g|_v})\otimes k=e_{g\cdot v}\otimes \pi_\tau(\delta_{g|_v})k;
\end{equation}
we set $U_g=\bigoplus U^j_g$. Since the $\{e_v\otimes k:v\in X^j\}$ are mutually orthogonal, there are isometries $S_x$ on $\HH_\pi$ such that $S_x(e_v\otimes k)=e_{xv}\otimes k$, and these isometries form a Toeplitz-Cuntz family. The following calculation using \eqref{charUj} shows that $U$ and $S$ satisfy property \eqref{ssTC} of Proposition~\ref{Toeplitz}:
\begin{align*}
S_{g\cdot x}U_{g|_x}(e_v\otimes k)
&=S_{g\cdot x}(e_{g|_x\cdot v}\otimes \pi_\tau(\delta_{(g|_x)|_v})k)
=e_{(g\cdot x)(g|_x\cdot v)}\otimes \pi_\tau(\delta_{g|_{xv}})k\\
&=e_{g\cdot (xv)}\otimes \pi_\tau(\delta_{g|_{xv}})k=U_g(e_{xv}\otimes k)=U_gS_x(e_v\otimes k).
\end{align*}
Now Proposition~\ref{Toeplitz} gives us a representation $\pi_{U,S}:\TT(G,X)\to B(\HH_\pi)$ such that $\pi_{U,S}(s_vu_gs_w^*)=S_vU_gS_w^*$.

We now take $\xi_\tau$ to be the canonical cyclic vector for the GNS representation $\pi_\tau$ (so that $\xi_\tau$ is the image in $\KK_\tau$ of the identity $1_{C^*(G)}$), and define $\psi_{\beta,\tau}:\TT(G,X)\to \C$ by
\begin{equation}\label{KMS_hopeful}
\psi_{\beta,\tau}(a) := (1-|X|e^{-\beta}) \sum_{j=0}^\infty e^{-\beta j}\Big(\sum_{z \in X^j} \big(\pi_{U,S}(a)(e_{z}\otimes \xi_\tau)\,|\,e_{z}\otimes\xi_\tau\big) \Big).
\end{equation} 
Since $\psi_{\beta,\tau}$ is a norm-convergent sum of vector states with non-negative coefficients, it is a positive functional; since $|X^j|=|X|^j$, summing the geometric series $\sum_j (|X|e^{-\beta})^j$ shows that $\psi_{\beta,\tau}(1)=1$, and $\psi_{\beta,\tau}$ is a state. 

To verify \eqref{repn_state}, we take $a=s_vu_gs_w^*$. Then
\begin{align}
\big(\pi_{\beta,\tau}(a)(e_{z}\otimes \xi_\tau)\,|\,e_{z}\otimes \xi_\tau\big)&=\big(S_v U_g S_w^*(e_{z}\otimes \xi_\tau)\,|\,e_{z}\otimes \xi_\tau\big)\notag\\
&=\big(U_gS_w^*(e_{z}\otimes \xi_\tau)\,|\,S_v^*(e_{z}\otimes \xi_\tau)\big).\label{spatialcalc}
\end{align}
We have $S_w^*(e_{z}\otimes \xi_\tau)=0$ unless $z=wz'$, and hence \eqref{spatialcalc} vanishes unless $z=wz'=vz''$, in which case
\begin{align*}
\big(\pi_{\beta,\tau}(a)(e_{z}\otimes \xi_\tau)\,|\,e_{z}\otimes \xi_\tau\big)&=\big(U_g(e_{z'}\otimes \xi_\tau)\,|\,e_{z''}\otimes \xi_\tau\big)\\
&=\big(e_{g\cdot z'}\otimes \pi_\tau(\delta_{g|_{z'}})\xi_\tau\,|\,e_{z''}\otimes \xi_\tau\big).
\end{align*}
This last inner product vanishes unless $g\cdot z'=z''$, which implies $|z'|=|z''|$ and $|v|=|z|-|z'|=|z|-|z''|=|w|$; now $z=wz'=vz''$ forces $v=w$ and $z'=z''$. Thus the inner product vanishes unless $z=vz'$ and $g\cdot z'=z'$. Noticing that $z=vz'$ implies $|z|\geq |v|$ and writing $y$ for $z'$, we find that
\begin{equation*}
\psi_{\beta,\tau}(s_vu_gs_w^*)=\begin{cases}
(1-|X|e^{-\beta}) \displaystyle{\sum_{j=|v|}^\infty e^{- \beta j} \Big(\sum_{\{y \in X^{j-|v|}\,:\, g \cdot y=y\}}} (\pi_\tau(\delta_{g|_{y}})\xi_\tau\,|\,\xi_\tau)\Big)& \text{if } v=w \\
0 & \text{otherwise.} \\
\end{cases}
\end{equation*}
Since $(\pi_\tau(\delta_{g|_{y}})\xi_\tau\,|\,\xi_\tau)=\tau(1_{C^*(G)}^*\delta_{g|_{y}}1_{C^*(G)})=\tau(\delta_{g|_{y}})$, taking $k=j-|v|$ gives \eqref{repn_state}.

We show that $\psi_{\beta,\tau}$ is a KMS$_\beta$ state by checking properties \eqref{char:eqn} and \eqref{phi_trace} of Proposition~\ref{KMS>beta}. The first is straightforward. We trivially have $\psi_{\beta,\tau}(s_vu_gs_w^*)=0$ if $v\not=w$. For $g\in G$ and $v\in X^j$, we have
\begin{align*}
e^{-\beta|v|}\psi_{\beta,\tau}(u_g)
&=(1-|X|e^{-\beta})e^{-\beta|v|} \displaystyle{\sum_{k=0}^\infty e^{-\beta k}\Big(\sum_{\{y \in X^{k}\,:\,g \cdot y=y\}}}  \tau(\delta_{g|_{y}})\Big),
\end{align*}
which on pulling $e^{-\beta |v|}$ inside the sum becomes the right-hand side of the formula \eqref{repn_state} for $\psi_{\beta,\tau}(s_vu_gs_v^*)$. For the second, we need to take $g,h\in G$ and compare
\begin{align}
\psi_{\beta,\tau}(u_gu_h)=\psi_{\beta,\tau}(u_{gh})&=(1-|X|e^{-\beta}) \displaystyle{\sum_{k=0}^\infty e^{-\beta k}\Big(\sum_{\{y \in X^{k}\,:\,(gh)\cdot y=y\}}}  \tau(\delta_{(gh)|_{y}})\Big)\label{1stsum}\\
\intertext{with}
\psi_{\beta,\tau}(u_hu_g)=\psi_{\beta,\tau}(u_{hg})&=(1-|X|e^{-\beta}) \displaystyle{\sum_{k=0}^\infty e^{-\beta k}\Big(\sum_{\{z \in X^{k}\,:\,(hg)\cdot z=z\}}}  \tau(\delta_{(hg)|_{z}})\Big).\label{2ndsum}
\end{align}
The function $f:X^k\to X^k$ defined by $f(y)=h\cdot y$ is a bijection, and
\[
(gh)\cdot y=y\Longleftrightarrow g\cdot(h\cdot y)=h^{-1}\cdot(h\cdot y)
\Longleftrightarrow (hg)\cdot(h\cdot y)=h\cdot y,
\]
so $f$ maps the index set $\{y \in X^{k}:(gh)\cdot y=y\}$ in \eqref{1stsum} onto the one $\{z \in X^{k}:(hg)\cdot z=z\}$ in \eqref{2ndsum}. We claim that the function $f$ also matches up the corresponding summands. To see this, suppose $(gh)\cdot y=y$. Then because $\tau$ is a trace, we have
\[
\tau(\delta_{(gh)|_y})=\tau(\delta_{g|_{h\cdot y}h|_y})=\tau(\delta_{h|_y}\delta_{g|_{h\cdot y}})=\tau(\delta_{h|_{h^{-1}\cdot(h\cdot y)}}\delta_{g|_{h\cdot y}}).
\]
The identity $(gh)\cdot y=y$ implies that $h^{-1}\cdot (h\cdot y)=g\cdot (h\cdot y)$, so 
\[
\tau(\delta_{(gh)|_y})=\tau(\delta_{h|_{g\cdot(h\cdot y)}}\delta_{g|_{h\cdot y}})=\tau(\delta_{(hg)|_{h\cdot y}})=\tau(\delta_{(hg)|_{f(y)}}),
\]
as claimed. We deduce that $\psi_{\beta,\tau}(u_gu_h)=\psi_{\beta,\tau}(u_hu_g)$, and now Theorem~5.1 implies that $\psi_{\beta,\tau}$ is a KMS$_\beta$ state.
\end{proof}

While we have the formulas for the induced representations handy, we describe the ground states and KMS$_\infty$ states of our system. 

\begin{prop}\label{ground}
Suppose that $(G,X)$ is a self-similar action. Then for every state $f$ of $C^*(G)$, there is a ground state $\phi_f$ on $(\TT(G,X),\sigma)$ such that
\begin{equation}\label{phitau}
\phi_f(s_vu_gs_w^*)=
\begin{cases} f(\delta_g)&\text{if $v=w=\varnothing$}\\
0&\text{otherwise.}\end{cases}
\end{equation}
The map $f\mapsto\phi_f$ is an affine homeomorphism of the state space $S(C^*(G))$ onto the ground states of $(\TT(G,X),\sigma)$. For $f\in S(C^*(G))$, $\phi_f$ is a KMS$_\infty$ state if and only if $f$ is a trace. 
\end{prop}

That states on $C^*(G)$ give ground states is proved in greater generality in \cite[Theorem~2.2]{ln}. However, as in Theorem~\ref{repn_hilbert}, we can use the special features of our situation to give specific formulas. 

We begin with an analogue of \cite[Lemma 8.4]{lr} which will allow us to recognise ground states. The proof of that lemma carries over almost verbatim to this situation.

\begin{lemma}\label{idground}
Suppose that $(G,X)$ is a self-similar action. A state $\phi$ of $\TT(G,X)$ is a ground state of $(\TT(G,X),\sigma)$ if and only if
\begin{equation}\label{ground_eqn}
\phi(s_v u_g s_w^*) = \begin{cases}
\phi(u_g) & \text{ if } v=w=\varnothing \\
0 & \text{ otherwise.} \\
\end{cases}
\end{equation}
\end{lemma}

\begin{proof}[Proof of Proposition~\ref{ground}]
Given a state $f$ of $C^*(G)$, we take the GNS representation $\pi_f$ of $C^*(G)$ on $\KK_f$ with cyclic vector $\xi_f$, and consider the representation $\pi_{U,S}$ of $\TT(G,X)$ on $\bigoplus_{j=0}^\infty M^{\otimes j}\otimes_{C^*(G)}\KK_f$, as in the proof of Theorem~\ref{repn_hilbert}. Then we define
\[
\phi_f(a)=\big(\pi_{U,S}(a)(e_\varnothing\otimes \xi_f)\,|\,e_\varnothing\otimes \xi_f\big)\quad\text{for $a\in \TT(G,X)$.}
\]
Then $\phi_f$ is a state, and 
\begin{align*}
\phi_f(s_vu_gs_w^*)&=\big(S_vU_gS_w^*(e_\varnothing\otimes \xi_f)\,|\,e_\varnothing\otimes \xi_f\big)\\
&=\begin{cases}
0&\text{unless $v=w=\varnothing$}\\
(\pi_f(\delta_g)\xi_f\,|\,\xi_f\big)=f(\delta_g)&\text{if $v=w=\varnothing$.}
\end{cases}
\end{align*}
Lemma~\ref{idground} implies that $\phi_f$ is a ground state. The map $f \mapsto \phi_f$ is continuous, affine and injective, and it is onto because $\phi=\phi_f$ for $f=\phi|_{C^*(G)}$. 

If $\phi$ is a KMS$_\infty$ state, then $\phi$ is the limit of a sequence of KMS$_\beta$ states, and equation \eqref{phi_trace} in Proposition~\ref{KMS>beta} implies that $\tau:=\phi|_{C^*(G)}$ is a trace. For the converse, suppose that $\tau$ is a trace on $C^*(G)$. Then we can use weak$^*$ compactness to get a sequence $\{\psi_{\beta_n,\tau}\}$ which converges to a KMS$_\infty$ state $\phi$. Since \eqref{char:eqn} gives $\psi_{\beta,\tau}(s_vu_gs_v^*)=e^{-\beta|v|}\psi_{\beta,\tau}(u_g)$, we have $\psi_{\beta,\tau}(s_vu_gs_v^*)\to 0$ as $\beta\to \infty$ whenever $|v|>0$. On the other hand, \eqref{repn_state} gives
\[
\psi_{\beta,\tau}(u_g)=(1-|X|e^{-\beta})\Big(\tau(\delta_g)+\sum_{k=1}^\infty e^{-k\beta} \Big(\sum_{\{y \in X^{k}\,:\,g \cdot y=y\}}  \tau(\delta_{g|_{y}})\Big).
\]
Now
\[
\Big|\sum_{k=1}^\infty e^{-k\beta} \Big(\sum_{\{y \in X^{k}\,:\,g \cdot y=y\}}  \tau(\delta_{g|_{y}})\Big)\Big|\leq |X|e^{-\beta}\sum_{j=0}^\infty |X|^je^{-\beta j} =\frac{|X|e^{-\beta}}{1-|X|e^{-\beta}}
\]
converges to $0$ as $\beta\to \infty$, and hence $\psi_{\beta,\tau}(u_g)\to\tau(\delta_g)$. Thus the limit $\phi$ is the state $\phi_\tau$ described in  \eqref{phitau}, and $\phi_\tau$ is KMS$_\infty$.
\end{proof}

In the situation of \cite{lrr}, where the group $G=\Z^d$ is abelian, every state on $C^*(G)$ is a trace, and we recover \cite[Proposition~8.1]{lrr}: every ground state of $(\TT(\Z^d,\Sigma),\sigma)$ is a KMS$_\infty$ state. For nonabelian $G$, though, there are many states of $C^*(G)$ which are not traces, and $(\TT(G,X),\sigma)$ has many ground states which are not KMS$_\infty$ states.

\section{Parameterisation of KMS$_\beta$ states on the Toeplitz algebra}\label{parameterise}

\begin{thm}\label{Conditioning}
Suppose that $(G,X)$ is a self-similar action  and $\beta > \log |X|$. The map $\tau\mapsto\psi_{\beta,\tau}$ in Theorem~\ref{repn_hilbert} is an affine homeomorphism from the simplex of normalised traces on the full group $C^*$-algebra $C^*(G)$ onto the simplex of KMS$_\beta$ states on $(\TT(G,X),\sigma)$.\end{thm}

For the proof, we need some lemmas. As in \cite[\S10]{lr} and \cite[\S7]{lrr}, the idea is to  show that a KMS$_\beta$ state can be reconstructed from its conditioning to a corner $P \TT(G,X) P$. Here we take
\[ P:= 1-\sum_{x \in X} s_x s_x^*\in \TT(G,X).\]

\begin{lemma}\label{condtrace}
Suppose that $\phi$ is a KMS$_\beta$ state,  and define $\phi_P:\TT(G,X)\to \C$ by
\[ \phi_P(a) = \frac{1}{1-|X|e^{-\beta}} \phi(P a P). \]
Then $\phi_P|_{C^*(G)}$ is a normalised trace. 
\end{lemma}

\begin{proof}
The function $\phi_P$ is a positive linear functional because $\phi$ is, and the computation
\begin{align*}
 \phi_P(1)&=  \frac{1}{1-|X|e^{-\beta}} \phi\Big(1-\sum_{x \in X} s_x s_x^*\Big)=\frac{1}{1-|X|e^{-\beta}}\Big(1-e^{-\beta}\sum_{x \in X}\phi(s_x^* s_x)\Big)\\
 &=\frac{1-|X|e^{-\beta}}{1-|X|e^{-\beta}}=1
 \end{align*}
shows that $\phi_P$ is a state. 

With a view to proving  that $\phi_P$ is tracial on $C^*(G)$, we claim that $u_g P=P u_g$. Indeed, for $x\in X$ and $g\in G$ we have
\[
u_gs_xs_x^*=u_gs_xs_x^*u_g^*u_g=(u_gs_x)(u_gs_x)^*u_g=(s_{g\cdot x}u_{g|_x})(s_{g\cdot x}u_{g|_x})^*u_g=s_{g\cdot x}s_{g\cdot x}^*u_g.
\]
Thus for $g\in G$, we have
\begin{align*}
u_g P &= u_g \Big(1-\sum_{x \in X} s_x s_x^*\Big) = u_g - \sum_{x \in X} u_g s_x s_x^*\\
&= u_g - \sum_{x \in X} s_{g \cdot x} s_{g \cdot x}^* u_g =\Big(1 - \sum_{x \in X} s_{g \cdot x} s_{g \cdot x}^*\Big) u_g,
\end{align*}
and since $1-\sum_{x \in X} s_{g \cdot x} s_{g \cdot x}^*=P$ we get $u_gP=Pu_g$, as claimed. Now since $\phi$ is a KMS$_\beta$ state, we have
\[
\phi(Pu_gu_hP)=\phi(u_gPu_h)=\phi(Pu_h\sigma_{i\beta}(u_g))=\phi(Pu_hu_g)=\phi(Pu_hu_gP),
\]
which implies that $\phi_P|_{C^*(G)}$ is a trace.
\end{proof}

\begin{lemma}\label{Key_lemma}
Let $(G,X)$ be a self-similar action.
For each $n \in \N$, the element
\[ p_n := \sum_{j=0}^n \sum_{v \in X^j} s_v P s_v^* \]
is a projection in $\TT(G,X)$, and if $\phi$ is a KMS$_\beta$ state of $(\TT(G,X),\sigma)$ and $a \in \TT(G,X)$, then $\phi(p_n a p_n) \rightarrow \phi(a)$ as $n \rightarrow \infty$.
\end{lemma}

\begin{proof}
Each $s_v P s_v^*$ is a projection, so we need to show that $s_v P s_v^*$ and $s_w P s_w^*$ are mutually orthogonal when $v \neq w$. Since $\{s_z:z\in X^m\}$ is a Toeplitz-Cuntz family for each $m$, this is trivially true for $|v|=|w|$. So suppose $|v|\not=|w|$. The product $Ps_v^*s_wP$ vanishes unless $v=wv'$ or $w=vw'$; since $(Ps_v^*s_wP)^*=Ps_w^*s_vP$, we may as well assume that $|w|> |v|$ and $w=vw'$. Then, writing $w'_1$ for the first letter in $w'$, we have
\begin{equation}\label{mutorthog}
P s_v^* s_w P=Ps_{w'}P=\Big(1-\sum_{x\in X}s_xs_x^*\Big)s_{w'}P=s_{w'}P-s_{w'_1}s_{w'_1}^*s_{w'}P=0.
\end{equation}
Thus each $p_n$ is a projection.

Lemma~7.3 of \cite{lrr} says that if $\phi$ is a state of a unital $C^*$-algebra $A$, and $\{p_n\}$ is a sequence of projections in $A$ such that $\phi(p_n) \rightarrow 1$, then $\phi(p_n a p_n) \rightarrow \phi(a)$ for every $a \in A$. So we aim to show that $\phi(p_n) \rightarrow 1$ as $n \rightarrow \infty$. The KMS condition gives
\begin{align*}
\phi(p_n)&= \sum_{j=0}^n \sum_{v \in X^j} \phi(s_v P s_v^*) = \sum_{j=0}^n \sum_{v \in X^j} e^{-\beta j}\phi(P) \\
&= \phi(P) \sum_{j=0}^n (|X|e^{-\beta})^j 
= (1-|X|e^{-\beta}) \sum_{j=0}^n (|X|e^{-\beta})^j,
\end{align*}
which converges to $1$ as $n\to \infty$. Thus the result follows from \cite[Lemma~7.3]{lrr}
\end{proof}

The following reconstruction formula is an analogue of \cite[Proposition 7.2]{lrr}.

\begin{lemma}\label{Conditioning_State}
Suppose $\beta > \log |X|$ and $\phi$ is a KMS$_\beta$ state on $\TT(G,X)$. Then for $a \in \TT(G,X)$,
\begin{equation}\label{Conditioning_definition}
\phi(a) =(1-|X|e^{-\beta}) \sum_{j=0}^\infty \sum_{v \in X^j} e^{-\beta j} \phi_P(s_v^* a s_v).
\end{equation}
\end{lemma}

\begin{proof}
Lemma \ref{Key_lemma} gives
\begin{align*}
\phi(a) = \lim_{n \to \infty} \phi(p_n a p_n)
&= \lim_{n \to \infty} \sum_{j=0}^n \sum_{l=0}^n \sum_{v \in X^j} \sum_{w \in X^l} \phi(s_v P s_v^* a s_w P s_w^*) \\
&= \lim_{n \to \infty} \sum_{j=0}^n \sum_{l=0}^n \sum_{v \in X^j} \sum_{w \in X^l} e^{-\beta j} \phi(P s_v^* a s_w P s_w^*s_v P) \\
&= \lim_{n \to \infty} \sum_{j=0}^n \sum_{v \in X^j} e^{-\beta j} \phi(P s_v^* a s_v P)\quad\text{(using \eqref{mutorthog}\,)}\\
&=(1-|X|e^{-\beta})\sum_{j=0}^\infty \sum_{v \in X^j} e^{-\beta j} \phi_P(s_v^* a s_v).\qedhere
\end{align*}
\end{proof}

\begin{proof}[Proof of Theorem \ref{Conditioning}]
By an application of the monotone convergence theorem, we can deduce from \eqref{repn_state} that $\tau\mapsto \psi_{\beta,\tau}$ is affine and weak$^*$ continuous. Since both sets of states are weak$^*$ compact, it suffices to show that $\tau\mapsto \psi_{\beta,\tau}$ is bijective.

To see injectivity, suppose that $\psi_{\beta,\tau}=\psi_{\beta,\rho}$, and take $g \in G$. Then the formula \eqref{repn_state} gives
\begin{align*}
\psi_{\beta,\tau}(u_g)&=(1-|X|e^{-\beta}) \sum_{j=0}^\infty \,\, \sum_{\{y \in X^{j}\,:\,g \cdot y=y\}} e^{-\beta j}\tau(\delta_{g|_{y}}) \\
&=(1-|X|e^{-\beta})\tau(\delta_g) +(1-|X|e^{-\beta})\sum_{k=0}^\infty \,\, \sum_{\{y \in X^{k+1}\,:\,g \cdot y=y\}} e^{-\beta (k+1)}\tau(\delta_{g|_{y}}).
\end{align*}
We can write the index set for the last sum as 
\[
\{y\in X^{k+1}:g\cdot y=y\}=\{xy':x\in X,\; y'\in X^k,\; g\cdot x=x,\; g|_x\cdot y'=y'\},
\]
and then another application of \eqref{repn_state} gives
\begin{align*}
\sum_{k=0}^\infty \,\, \sum_{\{y \in X^{k+1}\,:\,g \cdot y=y\}} e^{-\beta (k+1)}\tau(\delta_{g|_{y}})&=e^{-\beta}\sum_{k=0}^\infty \,\, \sum_{\{x \in X \,:\,g \cdot x=x\}} \,\, \sum_{\{y' \in X^{k}\,:\,g|_x \cdot y'=y'\}} e^{-\beta k}\tau(\delta_{(g|_x)|_{y'}})\\
&=\frac{e^{-\beta}}{1-|X|e^{-\beta}}\Big(\sum_{\{x \in X \,:\,g\cdot x=x\}} \,\, \psi_{\beta,\tau}(u_{g|_{x}})\Big).
\end{align*}
Thus 
\begin{equation}\label{tau}
\psi_{\beta,\tau}(u_g)=(1-|X|e^{-\beta})\tau(\delta_h) + e^{-\beta}\Big(\sum_{\{x \in X \,:\,g\cdot x=x\}} \,\, \psi_{\beta,\tau}(u_{g|_{x}})\Big).
\end{equation}
Similarly, we have
\begin{equation}\label{rho}
\psi_{\beta,\rho}(u_g) = (1-|X|e^{-\beta})\rho(\delta_g) + e^{-\beta}\Big(\sum_{\{x \in X \,:\,g\cdot x=x\}} \,\, \psi_{\beta,\rho}(u_{g|_{x}})\Big).
\end{equation}
Since $\psi_{\beta,\tau}= \psi_{\beta,\rho}$, subtracting \eqref{tau} from \eqref{rho} shows that $\tau(\delta_g)=\rho(\delta_g)$. Thus $\tau=\rho$, and $\tau\mapsto \psi_{\beta, \tau}$ is injective.

To see surjectivity, suppose that $\phi$ is a KMS$_\beta$ state on $\TT(G,X)$. Lemma~\ref{condtrace} implies that $\tau:=\phi_P|_{C^*(G)}$ is a normalised trace, and we aim to show that $\phi=\psi_{\beta,\tau}$. By \eqref{char:eqn}, it suffices to show that $\phi(u_g)=\psi_{\beta,\tau}(u_g)$ for all $g \in G$. Fix $g\in G$. Then the reconstruction formula \eqref{Conditioning_definition} gives
\begin{align}
\notag
\phi(u_g)&=(1-|X|e^{-\beta})\sum_{j=0}^\infty \sum_{y \in X^j} e^{-\beta j} \phi_P(s_y^* u_g s_y) \\
\notag
&=(1-|X|e^{-\beta})\sum_{j=0}^\infty \sum_{y \in X^j} e^{-\beta j} \phi_P(s_y^*  s_{g \cdot y} u_{g|_{y}}) \\
\notag
&=\lim_{n \to \infty} (1-|X|e^{-\beta})\sum_{j=0}^n \sum_{\{y \in X^j\,:\, g \cdot y=y\}} e^{-\beta j} \tau(u_{g|_{y}}),
\end{align}
 which by \eqref{repn_state} is precisely $\psi_{\beta,\tau}(u_g)$. Thus $\phi=\psi_{\beta,\tau}$, and $\tau\mapsto \psi_{\beta, \tau}$ is surjective.
\end{proof}

For every discrete group $G$, there are at least two normalised traces on $C^*(G)$. The usual trace $\tau_e$ on $C^*(G)$ satisfies
\[
\tau_e(\delta_g)=\begin{cases}
1&\text{if $g=e$}\\
0&\text{otherwise.}
\end{cases}
\]
To see that there is such a trace, consider the left-regular representation $\lambda$ of $G$ on $\ell^2(G)$, and define $\tau_e:C^*(G)\to \C$ in terms of the usual orthonormal basis $\{\xi_g:g\in G\}$ by $\tau_e(a)=(\lambda(a)\xi_e\,|\,\xi_e)$. Then it is easy to check on $\newspan\{\delta_g\}$ that $\tau_e$ has the required properties, and continuity does the rest. The other trace is the integrated form $\tau_1:C^*(G)\to \C$ of the trivial representation $g\mapsto 1$, which is a scalar-valued homomorphism, and hence is trivially a trace.

Since $\tau_e$ and $\tau_1$ do not agree on the $\delta_g$ with $g\not= e$, they are distinct traces, and hence by Theorem~\ref{Conditioning} give distinct KMS states. We look at these states.

\begin{cor}\label{usualtr}
Suppose that $(G,X)$ is a self-similar action and $\beta>\log|X|$. For $g\in G$ and $k\geq 0$, we set
\begin{equation}\label{defFgj}
F_g^k := \{v \in X^k : g \cdot v = v \text{ and } g|_v=e \}.
\end{equation}
Then there is a KMS$_\beta$ state $\psi_{\beta,\tau_e}$ on $(\TT(G,X),\sigma)$ such that
\[
\psi_{\beta,\tau_e}(s_v u_g s_w^*)=\begin{cases}
e^{-\beta|v|}(1-|X|e^{-\beta}) \displaystyle{\sum_{k=0}^\infty e^{-\beta k}|F^k_g|}& \text{if $v=w$}\\
0&\text{otherwise,}
\end{cases}
\]
where we interpret $|\varnothing|$ as $0$.
\end{cor}

\begin{proof}
The state $\psi_{\beta,\tau_e}$ is the one given by Theorem~\ref{repn_hilbert}, so we just need to check the formula for $\psi_{\beta,\tau_e}(s_v u_g s_w^*)$. It is certainly $0$ if $v\not= w$, so we suppose $v=w$. Then since $\tau_e(\delta_e)=1$ and $\tau_e(\delta_{h})=0$ for $h\not= e$, the sum on the right-hand side of \eqref{repn_state} collapses to give
\begin{equation*}
\psi_{\beta,\tau_e}(s_v u_g s_v^*)=
(1-|X|e^{-\beta}) \displaystyle{\sum_{k=0}^\infty e^{-\beta(k+|v|)} \Big(\sum_{y\in F_g^k}} 1\Big),
\end{equation*}
which on pulling out $e^{-\beta|v|}$ gives the required formula.
\end{proof}

\begin{cor}\label{trvialrep}
Suppose that $(G,X)$ is a self-similar action and $\beta>\log|X|$. For $g\in G$ and $k\geq 0$, we set
\begin{equation}\label{defGgj}
G_g^k := \{v \in X^k : g \cdot v = v\}.
\end{equation}
Then there is a KMS$_\beta$ state $\psi_{\beta,\tau_1}$ on $(\TT(G,X),\sigma)$ such that
\[
\psi_{\beta,\tau_1}(s_v u_g s_w^*)=\begin{cases}
e^{-\beta|v|}(1-|X|e^{-\beta}) \displaystyle{\sum_{k=0}^\infty e^{-\beta k}|G^k_g|}& \text{if $v=w$}\\
0&\text{otherwise.}
\end{cases}
\]
\end{cor}

\begin{proof}
As in the proof of the previous corollary, the second sum on the right-hand side of \eqref{repn_state} counts the number of elements of $G_g^k$, and hence this follows from Theorem~\ref{repn_hilbert}.
\end{proof}

Although the formulas in the last two corollaries look a bit messy, they are quite computable, and we will later discuss ways of doing these computations using Moore diagrams. But it is easy to give a quick example now.

\begin{example}\label{Toebas}
Consider the basilica group $(B,X=\{x,y\})$ of \S\ref{Basilica}. The first two relations in \eqref{basilica_action} imply that the generator $a$ changes the first letter of every word, so $F_a^k=G_a^k=\varnothing$ for every $k\geq 1$, and $\psi_{\beta,\tau_e}(\delta_a)=\psi_{\beta,\tau_1}(\delta_a)=0$. On the other hand, $b$ fixes $x$ with $b|_x=a$, and hence satisfies $b\cdot(xw)\not=xw$ for every longer word $xw$. Thus $F_b^k=\{yw:w\in X^*\}$ and $G_b^k=\{yw:w\in X^*\}\cup\{x\}$. We deduce that
\begin{align*}
\psi_{\beta,\tau_e}(\delta_b)&=(1-2e^{-\beta})\sum_{k=0}^\infty e^{-\beta k}2^{k-1}={\textstyle{\frac{1}{2}}}(1-2e^{-\beta})\sum_{k=0}^\infty (2e^{-\beta })^k={\textstyle\frac{1}{2}},\text{ and}\\
\psi_{\beta,\tau_1}(\delta_b)&=(1-2e^{-\beta})\sum_{k=0}^\infty e^{-\beta k}(2^{k-1}+1)={\textstyle\frac{1}{2}}+\frac{1-2e^{-\beta}}{1-e^{-\beta}}.
\end{align*}
\end{example}

\begin{remark}
When the group $G$ is abelian, the normalised traces on $C^*(G)\cong C(\hat G)$ are given by probability measures on the compact dual group $\hat G$. Thus in \cite{lrr} (see also \S\ref{checklrr} below), the KMS states with inverse temperature $\beta>\beta_c$ on $(\TT(\Z^d,\Sigma),\sigma)$ are parametrised by the probability measures on $\T^d$. 

When $G$ has an abelian quotient $Q$, $C^*(Q)$ is a quotient of $C^*(G)$, and the probability measures on $\hat Q$ give traces on $C^*(G)$ and KMS states on $(\TT(G,X),\sigma)$. This applies in particular to the self-similar action $(B,X)$ associated to the basilica group in \S\ref{Basilica}, since Proposition~\ref{quotB} implies that $B$ has a quotient isomorphic to $\Z^2$. Thus for each $\beta>\log|X|$, Theorem~\ref{Conditioning} gives a simplex $S_Q$ of KMS$_\beta$ states of $(\TT(B,X),\sigma)$ parametrised by the probability measures on $\hat Q=\T^2$. The simplex $S_Q$ includes the state $\psi_{\beta,\tau_1}$ of Corollary~\ref{trvialrep}, which corresponds to the point mass at $1\in \T^2$. However, since the trace $\tau_e$ does not factor through the quotient map, Theorem~\ref{Conditioning} implies that $S_Q$ does not include the state $\psi_{\beta,\tau_e}$ of Corollary~\ref{usualtr}.
\end{remark}

\section{KMS states at the critical inverse temperature}\label{KMS_CP_alg}

In this section we describe the KMS states on $\TT(G,X)$ at the critical inverse temperature $\beta_c=\log|X|$. We start by showing that we are effectively dealing with the KMS states on the Cuntz-Pimsner algebra $\OO(G,X)$.

\begin{prop}\label{factorthruCP}
Let $(G,X)$ be a self-similar action. Every KMS$_{\log |X|}$ state of $(\TT(G,X),\sigma)$ factors through a KMS$_{\log|X|}$ state on  $\OO(G,X)$.
\end{prop}

\begin{proof}
Suppose that $\phi$ is a KMS$_{\log |X|}$ state of $(\TT(G,X),\sigma)$. Then Proposition~\ref{KMS>beta} implies that $\phi(s_x s_x^*)=|X|^{-1}$, and hence
\[ \phi\Big(1- \sum_{x\in X} s_x s_x^*\Big) = 1-|X| |X|^{-1} = 0. \]
Now the argument in \cite[Lemma 10.3]{lr} implies that $\phi$ vanishes on the ideal $I$ generated by $1- \sum_{x\in X} s_x s_x^*$. (Or one could apply the more general result in \cite[Lemma~2.2]{hlrs} to the family $\FF=\{s_vu_gs_w^*\}$ of analytic elements.) Corollary \ref{CPquotient} says that $I$ is the kernel of the quotient map of $\TT(G,X)$ onto $\OO(G,X)$, and hence $\phi$ factors through this quotient map.
\end{proof}

To state our main results about states of $\OO(G,X)$, we need some information about the sets $F_g^k$ in Corollary~\ref{usualtr}.

\begin{prop}\label{Fng}
Suppose that $(G,X)$ is a self-similar action. For $g \in G \setminus \{e\}$ and $k\geq 0$, we consider again
\[
F_g^k = \{v \in X^k : g \cdot v = v \text{ and } g|_v=e \}.
\]
The sequence $\{|X|^{-k}|F_g^k|\}$ is increasing and converges with limit $c_g\in [0,1)$. 
\end{prop}

\begin{proof}
If $v \in F_g^k$ and $x \in X$, then 
\[
g \cdot (vx) = v(g|_v\cdot x) = vx\quad\text{ and}\quad g|_{vx} = (g|_v)|_x = e|_x = e,
\]
so $vx \in F_g^{k+1}$. Thus $|F_g^{k+1}|\geq |X|\,|F_g^k|$, and multiplying by $|X|^{-k-1}$  shows that $\{|X|^{-k}|F_g^k|\}$ is increasing.

Since the action of $G$ on $X^*$ is faithful, $g$ acts non-trivially on some $X^j$, say $g \cdot v \neq v$. Then $v$ is not in $F_g^j$, and no word of the form $vw$ is in any $F_g^l$. So for $k> j$,
\[
|F_g^k| \leq |X|^k - |X|^{k-j} = |X|^k(1-|X|^{-j}).
\]
Thus $|X|^{-k}|F_g^k|\leq 1-|X|^{-j}<1$ for $k>j$, and the sequence converges to a limit $c_g$ satisfying $c_g<1$.
\end{proof}

We can now state our main theorem about $\OO(G,X)$. Notice that part~\eqref{uniqueKMS} applies in particular when $(G,X)$ is contracting.

\begin{thm}\label{KMSatcritical}
Suppose that $(G,X)$ is a self-similar action.
\begin{enumerate}
\item \label{idinvtemp} Every KMS state of $(\OO(G,X),\sigma)$ has inverse temperature $\log|X|$.
\smallskip
\item\label{existcritKMS} Take $c_g$ as in Proposition~\ref{Fng}. Then there is a KMS$_{\log|X|}$ state on $\OO(G,X)$ such that
\begin{equation}\label{formcritstate}
\psi(s_vu_gs_w^*)=\begin{cases}
|X|^{-|v|}c_g&\text{if $v=w$}\\
0&\text{otherwise.}
\end{cases}
\end{equation}
\item\label{uniqueKMS} Suppose that for every $g\in G\setminus \{e\}$, the set $\{g|_v:v\in X^*\}$ is finite. Then the state in part \eqref{existcritKMS} is the only KMS state of $(\OO(G,X),\sigma)$.
\end{enumerate}
\end{thm}

\begin{proof}[Proof of Theorem~\ref{KMSatcritical}\,\eqref{idinvtemp}]
Suppose that $\phi$ is a KMS state of $(\OO(G,X),\sigma)$ with inverse temperature $\beta$. Then the Cuntz relation $\sum_{x\in X} s_x s_x^*=1$ and the KMS condition give
\begin{align*}
1=\phi(1)&=\phi \Big( \sum_{x\in X} s_x s_x^*\Big) =\sum_{x\in X} \phi(s_x s_x^*)=\sum_{x\in X} \phi(s_x^* \sigma_{i\beta}(s_x)) \\
&=\sum_{x\in X} e^{-\beta}\phi(s_x^* s_x) =\sum_{x\in X} e^{-\beta} = \d e^{-\beta},
\end{align*}
and hence $\beta = \log |X|$.
\end{proof}

We will prove existence of the KMS$_{\log |X|}$ state $\psi$ by taking a limit of KMS$_\beta$ states as $\beta\to \beta_c=\log|X|$. To evaluate the limit, we need the following analytic lemma.

\begin{lemma}\label{lem-cvgence}
Suppose that $\{c_k\}$ is an increasing sequence of real numbers with $c_k\to c$. Then 
\[
\sum_{k=0}^\infty (1-r)c_kr^k\to c\quad\text{as $r\to 1-$.}
\]
\end{lemma}

\begin{proof}
Fix $\epsilon>0$, and choose $K$ such that $k\geq K\Longrightarrow 0\leq c-c_k<\frac{\epsilon}{2}$. Choose $\delta>0$ such that 
\[
0<1-r<\delta\Longrightarrow \sum_{k=0}^K (1-r)c_kr^k<\textstyle{\frac{\epsilon}{2}}.
\]
Then for $r$ satisfying $0<1-r<\delta$, we have $\sum_{k=0}^\infty(1-r)r^k=1$, so
\begin{align*}
\Big|c-\sum_{k=0}^\infty (1-r)c_kr^k\Big|&=\Big|\sum_{k=0}^\infty (1-r)(c-c_k)r^k\Big|\\
&\leq \sum_{k=0}^K(1-r)(c-c_k)r^k
+(1-r)(c-c_K)\Big(\sum_{k=K+1}^\infty r^k\Big)\\
&=\sum_{k=0}^K(1-r)(c-c_k)r^k+(1-r)(c-c_K)r^K(1-r)^{-1},
\end{align*}
which is less than $\epsilon$ by choice of $K$ and $\delta$ (and because $r^K<1$).
\end{proof}

\begin{proof}[Proof of Theorem~\ref{KMSatcritical}\,\eqref{existcritKMS}]
We choose a decreasing sequence $\{\beta_n\}$ such that $\beta_n\to \log|X|$, and consider the KMS$_{\beta_n}$ states $\psi_{\beta_n}:=\psi_{\beta_n,\tau_e}$ of Corollary~\ref{usualtr}. By weak* compactness of the state space, we can by passing to a subsequence assume that $\{\psi_{\beta_n}\}$ converges weak* to a state $\psi$. Proposition~5.3.23 of \cite{bra-rob} implies that $\psi$ is a KMS$_{\log |X|}$ state. (Or we could wait till we have the formula \eqref{formcritstate}, and apply Proposition~\ref{KMS>beta}.)

We now compute the limit of $\psi_{\beta_n}(s_vu_gs_w^*)$. We know from \eqref{repn_state} that $\psi_{\beta_n}(s_vu_gs_w^*)=0$ unless $v=w$, and satisfies
\begin{align*} \label{averagedstate}
\psi_{\beta_n}(s_v u_g s_v^*)&=e^{-\beta_n|v|}(1-|X|e^{-\beta_n}) \sum_{k=0}^\infty  e^{-\beta_n k} |F^k_g|\\
&=e^{-\beta_n|v|}\Big(\sum_{k=0}^\infty (1-|X|e^{-\beta_n}) |X|^{-k}\,|F^k_g|(|X| e^{-\beta_n})^k\Big).
\end{align*}
Now we are in the situation of Lemma~\ref{lem-cvgence} with $r=|X|e^{-\beta_n}$ and $c_k=|X|^{-k}\,|F^k_g|\to c_g$. Since $r_n:=|X|e^{-\beta_n}$ converges to $1$ from below as $n\to \infty$, Lemma~\ref{lem-cvgence} implies that 
\[
\sum_{k=0}^\infty (1-r_n) c_kr_n^k\to c_g\quad\text{as $n\to\infty$.}
\]
Thus
\[
\psi(s_v u_g s_v^*)=\lim_{n\to \infty}\psi_{\beta_n}(s_v u_g s_v^*)=\lim_{n\to\infty} e^{-\beta_n|v|}\Big(\sum_{k=0}^\infty (1-r_n) c_kr_n^k\Big)=|X|^{-|v|}c_g,
\]
as required.
\end{proof}

\begin{proof}[Proof of Theorem~\ref{KMSatcritical}\,\eqref{uniqueKMS}]
Suppose that $\phi$ is a KMS state on $\OO(G,X)$. We need to show that $\phi$ is the state $\psi$ in \eqref{existcritKMS}. Part~\eqref{idinvtemp} implies that $\phi$ has inverse temperature $\log|X|$. Now Proposition~\ref{KMS>beta} implies that  it suffices for us to prove that $\phi(u_g)=\psi(u_g)$ whenever $g\not= e$.

Suppose that $g\in G\setminus\{e\}$. Since $\{g|_v:v\in X^*\}$ is finite and the action of $G$ on $X^*$ is faithful, there exists $j$ such that for each $v \in X^*$ with $g|_v\not= e$, there exists $u\in X^j$ with $g|_v \cdot u \neq u$. We will show that
\begin{equation}\label{tozero}
|X|^{-nj}\,\big|G^{nj}_g \setminus F_g^{nj}\big|=|X|^{-nj}\,\big|\{w \in X^{nj}: g \cdot w = w\} \setminus F_g^{nj}\big|\to 0\quad\text{ as $n\to \infty$,}
\end{equation}
and use this to show that $\phi(u_g)=c_g=\psi(u_g)$.

We prove by induction that
\begin{align}
\label{ind_E}
\big|G^{nj}_g \setminus F_g^{nj}\big| \leq (|X|^j -1)^n
\end{align}
for all $n \geq 1$. Our choice of $j$ ensures that, for every $v\in G^j_g$, the set $\{w \in X^j :g|_v \cdot w =w\}$ is not all of $X^j$; thus we have \eqref{ind_E} for $n=1$. Assume that \eqref{ind_E} holds for $n$. Then
\[
\big|G^{(n+1)j}_g \setminus F_g^{(n+1)j}\big| = \big|\{vv' :v \in X^{nj}, v' \in X^{j}, g \cdot vv' = vv'\}\setminus F_g^{(n+1)j}\big|, 
\]
and we have 
\[
vv'\in G^{(n+1)j}_g \setminus F_g^{(n+1)j}\Longrightarrow
v\in G^{nj}_g \setminus F_g^{nj}\text{ and }g|_v\cdot v'=v'.
\]
On the other hand, for each $v\in G^{nj}_g \setminus F_g^{nj}$, we have $g|_v\not=e$, and thus there exists $v'\in X^j$ such that $g|_v\cdot v'\not= v'$. Thus for each $v\in G^{nj}_g \setminus F_g^{nj}$,
\[
\big|\big\{v':vv'\in G^{(n+1)j}_g \setminus F_g^{(n+1)j}\big\}\big|\leq |X|^j-1,
\]
and the inductive hypothesis gives
\[
\big|G^{(n+1)j}_g \setminus F_g^{(n+1)j}\big|\leq \big|G^{nj}_g \setminus F_g^{nj}\big|\big(\,|X|^j-1\big) \leq \big(|X|^j-1\big)^{n+1}.
\]
Thus \eqref{ind_E} holds for all $n \geq 1$. Now we have
\[
0\leq|X|^{-nj}\,\big|G^{nj}_g \setminus F_g^{nj}\big| \leq |X|^{-nj}(|X|^j -1)^n = \Big(1-\frac{1}{|X|^j}\Big)^n \to 0\quad\text{ as $n\to \infty$,}
\]
which gives \eqref{tozero}.

To complete the proof we show that $\phi(u_g)=c_g$. For every $n \in \N$, we use the Cuntz relation $1=\sum_{w \in X^{nj}} s_w s_w^*$  and Proposition~\ref{KMS>beta} to compute
\begin{align}
\notag
\phi(u_g) &= \phi\Big( u_g \sum_{w \in X^{nj}} s_{w} s_w^* \Big) \\
\notag
&= \sum_{w \in X^{nj}} \phi(s_{g \cdot w} u_{g|_w} s_w^*) \\
\notag
&= \sum_{\{w \in X^{nj} \,:\,g \cdot w = w\}} |X|^{-nj} \phi(u_{g|_w}) \\
\notag
& = \sum_{w \in G^{nj}_g \setminus F_g^{nj}} |X|^{-nj} \phi(u_{g|_w}) +  \sum_{w \in F_g^{nj}} |X|^{-nj} \phi(u_{e}) \\
\label{form_ug}
& = \sum_{w \in G^{nj}_g \setminus F_g^{nj}} |X|^{-nj} \phi(u_{g|_w}) +  |X|^{-nj} |F_g^{nj}|.
\end{align}
Let $\varepsilon > 0$. By Proposition~\ref{Fng}, there exists $N \in \N$ such that 
\[
n \geq N\Longrightarrow \big|c_g - |X|^{-nj} |F_g^{nj}|\,\big|< \varepsilon/2\quad \text{and}\quad \Big(1-\frac{1}{|X|^j}\Big)^n < \varepsilon/2.
\]
Then for $n \geq N$, \eqref{form_ug} gives
\begin{align*}
| \phi(u_g) - c_g| &< \sum_{w \in G^{nj}_g \setminus F_g^{nj}} |X|^{-nj} |\phi(u_{g|_w})| + \varepsilon/2 \\
& \leq \big|G^{nj}_g \setminus F_g^{nj}\big| \, |X|^{-nj} + \varepsilon/2 \\
& \leq \Big(1-\frac{1}{|X|^j}\Big)^n + \varepsilon/2 < \varepsilon,
\end{align*}
which implies that $\phi(u_g)=c_g$.
\end{proof}

Somewhat surprisingly, our construction of KMS states at the critical inverse temperature gives a third trace on $C^*(G)$.

\begin{cor}\label{newtr}
Suppose that $(G,X)$ is a self-similar action, and take $\{c_g\}$ as in Proposition~\ref{Fng}. Then there is a trace $\tau$ on $C^*(G)$ such that $\tau(\delta_g)=c_g$ for $g\not= e$.
\end{cor}

\begin{proof}
Proposition~\ref{KMS>beta}\,\eqref{idb} implies that $\phi\circ \pi_u$ is a trace on $C^*(G)$ for every KMS state $\phi$ of $\TT(G,X)$ or $\OO(G,X)$, and taking $\phi$ to be the KMS$_{\log|X|}$ state of $\OO(G,X)$ in Theorem~\ref{KMSatcritical}\,\eqref{existcritKMS} gives the required trace $\tau:=\phi\circ\pi_u$.
\end{proof}

It will follow from Propositions~\ref{basilicatr} and~\ref{Grigonnuc} below that, for the self-similar actions of the basilica and Grigorchuk groups, the trace of Corollary~\ref{newtr} is distinct from the traces $\tau_e$ and $\tau_1$ considered in \S\ref{parameterise}. 

\begin{remark}
In \cite[\S3.4]{Pla}, Planchat constructs a trace $\Tr$ on a quotient $C^*_{\rho}(G)$ of $C^*(G)$, and the value $\Tr(\rho_g)$ at a unitary generator is (in our notation) the limit $\lim_{k\to\infty}|X|^{-k}|G^k_g|$ of the decreasing sequence $\{|X|^{-k}|G^k_g|\}$. When $(G,X)$ has the finite-state property of Theorem~\ref{KMSatcritical}\,\eqref{uniqueKMS}, the calculation \eqref{tozero} implies that $|X|^{-k}|G^k_g|\to c_g$ also, and hence our trace coincides with the lift of Planchat's trace to $C^*(G)$. For the groups generated by automata studied in \cite{Pla}, the pair $(G,X)$ always has this finite-state property. (To see this, note that $G$ is generated by a finite set $S$ which is closed under restriction. This generating family induces a length function $l$ on $G$, and then the properties of restriction imply that $l(g|_v)\leq l(g)$ for all $g\in G$ and $v\in X^*$. Since there are finitely many words of a fixed length, it follows that each $\{g|_v:v\in X^*\}$ is finite.) 

Our calculations in the next section suggest that it may be easier to compute the values of this trace using the formula $c_g=\lim_{k\to \infty}|X|^{-k}|F^k_g|$.
\end{remark}

We finish by showing that for a contracting self-similar action, the values of $c_g$ on the nucleus determine the function $c$, and hence the KMS state at critical inverse temperature. For convenience, we define $c_e:=1$.

\begin{cor}\label{nucenough}
Suppose that $(G,X)$ is a contracting self-similar action with nucleus $\NN$. For $g\in G$, choose $k\in \N$ such that $g|_w\in \NN$ for every $w\in X^k$. Then
\[
c_g=\sum_{\{w\in X^k\,:\,g\cdot w=w\}}|X|^{-k}c_{g|_w}.
\]
\end{cor}

\begin{proof}
We let $\phi$ be the unique KMS$_{\log|X|}$ state of $(\OO(G,X),\sigma)$, so that in particular $\phi(u_g)=c_g$ for all $g$ (see Theorem~\ref{KMSatcritical}\,\eqref{existcritKMS}). Now the result follows from the calculation in the first three lines of \eqref{form_ug}.
\end{proof}

\section{Examples}

\subsection{Dilation matrices}\label{checklrr}
Suppose that $A\in M_d(\Z)$ has $|\det A|>1$, and consider the associated self-similar action $(\Z^d,\Sigma)$ of \S\ref{dilation_example}. We first check that the states constructed in Theorem~\ref{repn_hilbert} are the same as the ones in \cite[Proposition~6.1]{lrr}.
 
The Fourier  transform gives an isomorphism of $C^*(G)=C^*(\Z^d)$ onto $C(\T^d)$; we choose the one which takes $\delta_n$ to the function $z\mapsto z^n$. Traces on $C^*(\Z^d)$ are given by probability measures on $\T^d$; given such a measure $\mu$, we consider the trace $\tau_\mu$ such that $\tau_\mu(\delta_n)=\int_{\T^d}z^n\,d\mu(z)$. We want to compute the values of the state $\psi_{\beta,\tau_\mu}$ of Theorem~\ref{repn_hilbert} on an element $s_wu_ns_w^*$ (it vanishes on the other spanning elements). For $j\geq 0$ and $u\in \Sigma^j$, we have
\[
n\cdot u=u\Longleftrightarrow n+b_j(u)=b_j(u)\Longleftrightarrow n\in B^j\Z^d,
\]
so $\{u\in \Sigma^j:n\cdot u=u\}$ is either $\Sigma^j$ (when $n\in B^j\Z^d$) or empty. If $n\cdot u=u$, then $n|_u=B^{-j}n$ by \eqref{idrest}, so the right-hand side of \eqref{repn_state} is
\begin{equation}\label{RHS5*}
(1-|\det A|e^{-\beta}) {\sum_{\{j\geq 0  \mid n\in B^j\Z^d\}}} |\det A|^je^{-\beta(|w|+j)}\tau_\mu(\delta_{B^{-j}n}).
\end{equation}
Thus  we have $n\in B^j\Z^d\Longleftrightarrow B^{|w|}n\in B^{|w|+j}\Z^d$, and writing $j'=|w|+j$ in \eqref{RHS5*} gives
\begin{equation}\label{eqfromlrr}
\psi_{\beta,\tau_\mu}(s_wu_ns_w^*)=
(1-|\det A|e^{-\beta}) {\sum_{\{j'\geq |w| \mid B^{|w|}n\in B^{j'}\Z^d\}}} |det A|^{j'-|w|}e^{-\beta j'} \int_{\T^d}z^{B^{(|w|-j')}n}\,d\mu(z).
\end{equation}

The isomorphism $\theta:\TT(\Z^d,\Sigma)\to \TT(M_L)$ of Proposition~\ref{isoToeplitz} carries an element $s_wu_ns_w^*$ into $u_{b(w)+B^{|w|}n}v^{|w|}v^{*|w|}u_{b(w)}^*$, and  we can check that the right-hand side of \eqref{eqfromlrr} is the same as the value of the state $\psi_{\beta,\mu}$ of \cite[Proposition~6.1]{lrr} on the spanning element $u_{b(w)+B^{|w|}n}v^{|w|}v^{*|w|}u_{b(w)}^*$.

\begin{prop}[{\cite[Theorem~5.3]{lrr}}]\label{fromlrr}
Suppose that $A\in M_d(\Z)$ has $N:=|\det A|\not=0$. Then there is a KMS$_{\log N}$ state $\phi$ of $(\OO(M_L),\sigma)=C^*(u,v)$ such that 
\begin{equation}\label{KMSlogNdilation}
\phi(u_mv^kv^{*l}u_n^*)=
\begin{cases}0&\text{unless $k=l$ and $m=n$}\\
N^{-k}&\text{if $k=l$ and $m=n$.}
\end{cases}
\end{equation} 
If $A$ is a dilation matrix, then this is the only KMS state of $(\OO(M_L),\sigma)$.
\end{prop}

\begin{proof}
To apply Theorem~\ref{KMSatcritical}\,\eqref{existcritKMS} to the associated self-similar group $(\Z^d,\Sigma)$, we need to compute the numbers $|F^j_n|$. For $u\in \Sigma^j$, we have $n\cdot u=u\Longleftrightarrow n\in B^j\Z^d$, and then $n|_u=B^{-j}n$, so $n|_u=0\Longleftrightarrow n=0$. Thus $F^j_n=\varnothing$ for all $n\not=0$, and the state $\psi$ of Theorem~\ref{KMSatcritical}\,\eqref{existcritKMS} satisfies
\begin{equation}\label{formcritstate2}
\psi(s_vu_ns_w^*)=\begin{cases}0&\text{unless $v=w$ and $n=0$}\\
|\Sigma|^{-|w|}=N^{-|w|}&\text{if $v=w$ and $n=0$.}
\end{cases}
\end{equation}
We take $\phi:=\psi\circ\theta^{-1}$. Then the elements of the form $\theta(s_ws_w^*)$ are the $u_mv^kv^{*l}u_n^*$ with $k=l=|w|$ and $m=n=b(w)$, and \eqref{formcritstate2} reduces to the formula \eqref{KMSlogNdilation} for $\phi$.

Now suppose that $A$ is a dilation matrix. Then Proposition~\ref{prop:dilation} implies that $(\Z^d,\Sigma)$ is a contracting self-similar action, and the uniqueness follows from Theorem~\ref{KMSatcritical}\,\eqref{uniqueKMS}.
\end{proof}

\subsection{Computing using the Moore diagram}\label{sec:Mooreuseful}

To calculate values of the KMS states explicitly, we need to compute the sizes of the sets $F_g^k$ and $G_g^k$ defined in \eqref{defFgj} and \eqref{defGgj}. We begin with $G_g^k$.

For each $v\in G_g^k$ we get the following path $\mu_v$ in the Moore diagram:
\begin{center}
\begin{tikzpicture}
\node[vertex] (vertexe) at (-5.7,0)   {$\mu_v:=$};
\node[vertex] (vertexe) at (-5,0)   {$g$};
\node[vertex] (vertexa) at (-2.5,0)   {$g|_{v_1}$}	
	edge [<-,>=latex,out=180,in=0,thick] node[auto,swap,pos=0.5]{$\scriptstyle(v_1,v_1)$} (vertexe);
\node[vertex] (vertexb) at (0,0)  {$g|_{v_1v_2}$}
	edge [<-,>=latex,out=180,in=0,thick] node[auto,swap,pos=0.5]{$\scriptstyle(v_2,v_2)$}(vertexa);
\node[vertex] (vertexc) at (2.5,0)   {$\cdots$}
	edge [<-,>=latex,out=180,in=0,thick] node[auto,swap,pos=0.5]{$\scriptstyle(v_3,v_3)$} (vertexb);
\node[vertex] (vertexd) at (5,0)  {$g|_{v}$}
	edge [<-,>=latex,out=180,in=0,thick] node[auto,swap,pos=0.5]{$\scriptstyle(v_k,v_k)$}(vertexc);
\end{tikzpicture}
\end{center}
Notice that all the labels have the form $(x,x)$. Every path with labels $(x,x)$ arises this way: given
\begin{center}
\begin{tikzpicture}
\node[vertex] (vertexe) at (-5.6,0)   {$\mu:=$};
\node[vertex] (vertexe) at (-5,0)   {$g$};
\node[vertex] (vertexa) at (-2.5,0)   {$h_1$}	
	edge [<-,>=latex,out=180,in=0,thick] node[auto,swap,pos=0.5]{$\scriptstyle(x_1,x_1)$} (vertexe);
\node[vertex] (vertexb) at (0,0)  {$h_2$}
	edge [<-,>=latex,out=180,in=0,thick] node[auto,swap,pos=0.5]{$\scriptstyle(x_2,x_2)$}(vertexa);
\node[vertex] (vertexc) at (2.5,0)   {$\cdots$}
	edge [<-,>=latex,out=180,in=0,thick] node[auto,swap,pos=0.5]{$\scriptstyle(x_3,x_3)$} (vertexb);
\node[vertex] (vertexd) at (5,0)  {$h_k,$}
	edge [<-,>=latex,out=180,in=0,thick] node[auto,swap,pos=0.5]{$\scriptstyle(x_k,x_k)$}(vertexc);
\end{tikzpicture}
\end{center}
we have $h_i=(\cdots((g|_{x_1})|_{x_2})\cdots)|_{x_{i}}=g|_{x_1\cdots x_{i}}$, and $v=x_1x_2\cdots x_k$ belongs to $G_g^k$ with $\mu_v=\mu$. We call paths $\mu$ of this form \emph{stationary}, because they give elements $v$ of $X^*$ such that $s(\mu)\cdot v=v$, where $s(\mu)\in G$ is the source of the path $\mu$. Thus $G_g^k$ is in one-to-one correspondence with the set of stationary paths in the Moore diagram starting at $g$. 

For $v\in F_g^k$, we have $g\cdot v=v$ and $g|_v=e$, so the last vertex on $\mu_v$ is $e$. Thus the elements of $F_g^k$ are in one-to-one correspondence with the stationary paths starting at $g$ and ending at $e$.

Thus we can compute $|G_g^k|$ and $|F_g^k|$ by counting stationary paths in the Moore diagram. Notice that for a given $g$, we only need to draw the part of the Moore diagram which consists of the stationary edges reachable by stationary paths from $g$. For examples of such computations, see Examples~\ref{sampleBas} and~\ref{sampleGrig} below.

\subsection{The basilica group}

We now consider the self-similar action $(B,X)$ which defines the basilica group (see \S\ref{Basilica}). In Example \ref{Toebas}, we discussed KMS states on the Toeplitz system $(\TT(B,X),\sigma)$ at inverse temperatures greater than the critical value $\beta_c=\log|X|=\log2$. At the critical inverse temperature, Proposition~\ref{factorthruCP} implies that every KMS$_{\log 2}$ state factors through $(\OO(B,X),\sigma)$, and we have:

\begin{prop}\label{basilicatr}
The system $(\OO(B,X),\sigma)$ has a unique KMS$_{\log 2}$ state, which is given on the nucleus $\NN = \{ e,a,b,a^{-1},b^{-1},ab^{-1},ba^{-1}\}$ by
\[ 
\phi(u_g) = \begin{cases}
1 & \text{ for } g=e \\
\frac{1}{2} & \text{ for } g=b,b^{-1} \\
0 & \text{ for } g = a,a^{-1},ab^{-1},ba^{-1}. 
\end{cases}
\]
\end{prop}

\begin{proof}
We know from Proposition~\ref{Bascontract} that $(B,X)$ is contracting with nucleus $\NN$, so existence and uniqueness of $\phi$ follow from Theorem~\ref{KMSatcritical}. In Figure~\ref{fig:Basilica}, there are no stationary paths from $g\in \{a,a^{-1},ab^{-1},ba^{-1}\}$ to $e$, so for such $g$ we have $F_g^k=\varnothing$ for all $k$ and $\phi(u_g)=c_g=0$. For $g \in \{b,b^{-1}\}$, the only stationary paths go straight  from $g$ to $e$, and there are $2^{k-1}$ of them; thus $|X|^{-k}|F_g^k|=2^{-k}2^{k-1}=\frac{1}{2}$, and $\phi(u_g)=c_g=\frac{1}{2}$.
\end{proof}

Corollary~\ref{nucenough} implies that these computations of the KMS$_{\log 2}$ state on the nucleus suffice to determine the state. If we want to know other values of the state, we can use the strategy outlined in \S\ref{sec:Mooreuseful}. We illustrate this strategy by calculating $\phi(s_v u_{aba}s_w^*)$.

\begin{example}\label{sampleBas} 
By Theorem~\ref{KMSatcritical}\,\eqref{existcritKMS}, $\phi(s_vu_{aba}s_w^*)$ is either $0$ or $2^{-|v|}\phi(u_{aba})$. To compute $\phi(u_{aba})$, we draw the portion of the Moore diagram emanating from $aba$ with a view to finding $F_{aba}^k$.  From the defining relations, we  calculate
\begin{alignat*}{2}
aba\cdot x&= x&\qquad  (aba)|_x&= (ab)|_{a\cdot x} a|_x = (ab)|_y b = a|_y eb = b \\
aba\cdot y&= y & \qquad (aba)|_y&= (ab)|_{a\cdot y} a|_x = (ab)|_x e = a|_{b\cdot x} a = a|_{x} a =ba. 
\end{alignat*}
We then note that $ba\cdot x=b\cdot y=y$, which forces $ba\cdot y=x$, and hence there are no stationary paths going from $aba$ to $e$ through $ba$. Now we can delete any edges in the Moore diagram for the nucleus with unequal labels, and find that all the stationary paths from $aba$ to $e$ lie in  the diagram
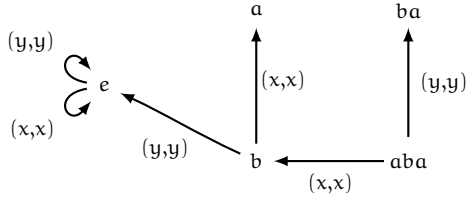
\begin{center}
\begin{tikzpicture}
\node at (-2,1) {$\scriptstyle e$};
\node[vertex] (vertexe) at (-2,1)   {$\,$}
	edge [->,>=latex,out=170,in=130,loop,thick] node[auto,pos=0.5]{$\scriptstyle (y,y)$} (vertexe)
	edge [->,>=latex,out=190,in=230,loop,thick] node[auto,swap,pos=0.5]{$\scriptstyle (x,x)$} (vertexe);
\node at (0,0) {$\scriptstyle b$};
\node[vertex] (vertexb) at (0,0)   {$\,$}	
	edge [->,>=latex,out=155,in=335,thick] node[auto,xshift=0.2cm,pos=0.45]{$\scriptstyle(y,y)$} (vertexe);
\node at (0,2) {$\scriptstyle a$};
\node[vertex] (vertexa) at (0,2)  {$\,$}
	edge [<-,>=latex,out=270,in=90,thick] node[auto,xshift=-0.1cm,pos=0.45]{$\scriptstyle(x,x)$}(vertexb);
\node at (2,2) {$\scriptstyle ba$};
\node[vertex] (vertexc) at (2,2)   {$\,$};
\node at (2,0) {$\scriptstyle aba$};
\node[vertex] (vertexd) at (2,0)  {$\,\,\,\,$}
	edge [->,>=latex,out=90,in=270,thick] node[auto,swap,pos=0.5]{$\scriptstyle(y,y)$}(vertexc)
	edge [->,>=latex,out=180,in=0,thick] node[auto,xshift=0.15cm,pos=0.6]{$\scriptstyle(x,x)$} (vertexb);
\end{tikzpicture}
\end{center}
We deduce that $|F_{aba}^k|=2^{k-2}$ for $k\geq 2$, and hence 
\[
\phi(u_{aba})=c_{aba}=\lim_{k\to \infty} 2^{-k}|F_{aba}^k|={\textstyle{\frac{1}{4}}}. 
\]
Thus Theorem~\ref{KMSatcritical}\,\eqref{existcritKMS} gives
\begin{equation*}
\phi(s_vu_{aba}s_w^*)=\begin{cases}
2^{-|v|-2}&\text{if $v=w$}\\
0&\text{otherwise.}
\end{cases}
\end{equation*}
\end{example}

\subsection{The Grigorchuk group}\label{Grigorchuk_state}

\begin{prop}\label{Grigonnuc}
Let $(G,X)$ be the self-similar action of the Grigorchuk group from \S\ref{Grigorchuk}. Then $(\OO(G,X),\sigma)$ has a unique KMS$_{\log 2}$ state $\phi$ which is given on the nucleus $\NN = \{e,a,b,c,d\}$ by
\[ 
\phi(u_g) = \begin{cases}
1 & \text{ for } g=e\\
0 & \text{ for } g = a \\
1/7 & \text{ for } g=b \\
2/7 & \text{ for } g=c \\
4/7 & \text{ for } g=d.
\end{cases}
\]
\end{prop}

\begin{proof}
We know from Proposition~\ref{grig:nuc} that the Grigorchuk action is contracting with nucleus $\NN$, and $|X|=2$, so Theorem~\ref{KMSatcritical}\,\eqref{uniqueKMS} implies that there is a unique KMS$_{\log 2}$ state $\phi$. A look at the Moore diagram shows that there are no stationary paths starting at $a$, and hence there are no stationary paths going to $e$ through $a$. Thus it suffices to count paths to $e$ in the following diagram. 
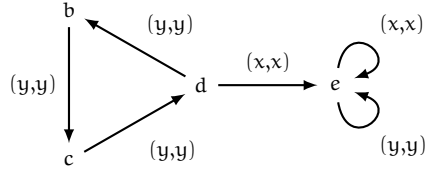
\begin{center}
\begin{tikzpicture}
\node at (0,0) {$\scriptstyle e$};
\node[vertex] (vertexe) at (0,0)   {$\,$}
	edge [->,>=latex,out=280,in=340,loop,thick] node[auto,swap,pos=0.5]{$\scriptstyle (y,y)$} (vertexe)
	edge [->,>=latex,out=80,in=20,loop,thick] node[auto,pos=0.5]{$\scriptstyle (x,x)$} (vertexe);
\node at (-3.5,-1) {$\scriptstyle c$};
\node[vertex] (vertexb) at (-3.5,-1)   {$\,$};
\node at (-3.5,1) {$\scriptstyle b$};
\node[vertex] (vertexc) at (-3.5,1)   {$\,$}
	edge [->,>=latex,out=270,in=90,thick] node[auto,swap,pos=0.5]{$\scriptstyle(y,y)$} (vertexb);
\node at (-1.7679,0) {$\scriptstyle d$};
\node[vertex] (vertexd) at (-1.7679,0)  {$\,$}
	edge [->,>=latex,out=0,in=180,thick] node[auto,pos=0.5]{$\scriptstyle(x,x)$} (vertexe)
	edge [->,>=latex,out=150,in=330,thick] node[auto,swap,pos=0.5]{$\scriptstyle(y,y)$}(vertexc)
	edge [<-,>=latex,out=210,in=30,thick] node[auto,xshift=0.15cm,pos=0.6]{$\scriptstyle(y,y)$} (vertexb);
\end{tikzpicture}
\end{center}
In particular, $F^k_a=\varnothing$ for all $k$, and $\phi(u_a)=c_a=0$. From $d$, there are $2^{k-1}$ paths of length $k$ which go straight to $e$, $2^{k-4}$ which first go round the cycle once, and
\[
|F^k_d|=2^{k-1}+2^{k-4}+\cdots+2^{k-(3j+1)}\quad\text{where $3j+1\leq k\leq 3j+3$.}
\]
Summing the geometric series gives
\[
|F_d^k|=2^{k-(3j+1)}\Big(\frac{(2^{3})^{(j+1)}-1}{2^3-1}\Big)=\frac{2^{k+2}-2^{k-(3j+1)}}{7}\quad\text{where $3j+1\leq k\leq 3j+3$.}
\]
Thus
\[
|X|^{-k}\,|F_d^k|=2^{-k}|F_d^k|=\frac{4-2^{-(3j+1)}}{7}\quad\text{where $3j+1\leq k\leq 3j+3$,}
\]
and $\phi(u_d)=c_d=\lim_{k\to\infty}|X|^{-k}\,|F_d^k|=\frac{4}{7}$. There are similar formulas for $c$ and $b$:
\begin{align}\label{formFc}
|F_c^k|&=|F_d^{k-1}|=\frac{2^{k+1}-2^{k-(3j+2)}}{7}\quad\text{where $3j+2\leq k\leq 3j+4$, and}\\
|F_b^k|&=|F_d^{k-2}|=\frac{2^{k}-2^{k-(3j+3)}}{7}\quad\text{where $3j+3\leq k\leq 3j+5$,}\notag
\end{align}
and these formulas imply that $\phi(u_c)=c_c=\frac{2}{7}$ and $\phi(u_b)=c_b=\frac{1}{7}$.
\end{proof}

\begin{example}\label{sampleGrig}
We calculate the value of the state $\phi$ in Proposition~\ref{Grigonnuc} on the generator $u_{cadac}$. We need the part of the Moore diagram emanating from $cadac$ with stationary edges. We calculate, using either the defining relations \eqref{grig_action} or the same information encoded in the Moore diagram of Figure~\ref{fig:Grig}:
\begin{alignat*}{2}
cadac\cdot x &= x &\qquad (cadac)|_x&= (cada)|_{x} a = (cad)|_y a = (ca)|_y ba = c|_x ba = aba \\
cadac\cdot y&= y &\qquad (cadac)|_y&= (cada)|_{y} d = (cad)|_x d = (ca)|_{x} d = c|_{y} d =d^2 = e \\
aba \cdot x&= x &\qquad (aba)|_x&=(ab)|_y=a|_{y} c = c \\
aba \cdot y&= y &\qquad (aba)|_y&=(ab)|_{x}  = a|_x a = a.
\end{alignat*}
This gets us into the nucleus. Now adding the stationary edges from the Moore diagram of the nucleus gives a diagram which contains all the stationary paths from $cadac$ to $e$:
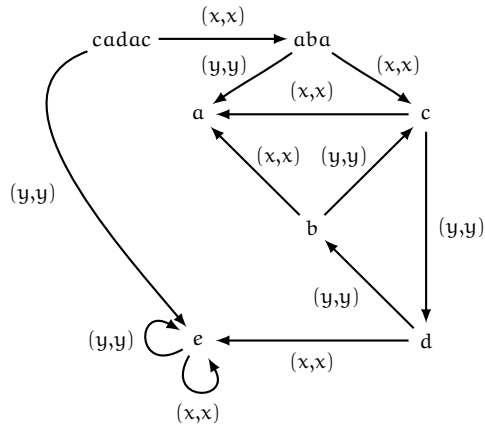
\begin{center}
\begin{tikzpicture}
\node at (0,0) {$\scriptstyle e$};
\node[vertex] (vertexe) at (0,0)   {$\,$}
	edge [->,>=latex,out=210,in=150,loop,thick] node[auto,pos=0.5]{$\scriptstyle (y,y)$} (vertexe)
	edge [->,>=latex,out=240,in=300,loop,thick] node[auto,swap,pos=0.5]{$\scriptstyle (x,x)$} (vertexe);
\node at (1.5,1.5) {$\scriptstyle b$};
\node[vertex] (vertexb) at (1.5,1.5)   {$\,$};
\node at (0,3) {$\scriptstyle a$};
\node[vertex] (vertexa) at (0,3)  {$\,$}
		edge [<-,>=latex,out=315,in=135,thick] node[auto,xshift=-0.25cm,pos=0.6]{$\scriptstyle(x,x)$} (vertexb);
\node at (3,3) {$\scriptstyle c$};
\node[vertex] (vertexc) at (3,3)   {$\,$}
	edge [<-,>=latex,out=225,in=45,thick] node[auto,swap,xshift=0.25cm,pos=0.6]{$\scriptstyle(y,y)$} (vertexb)
	edge [->,>=latex,out=180,in=0,thick] node[auto,swap,pos=0.5]{$\scriptstyle(x,x)$} (vertexa);
\node at (3,0) {$\scriptstyle d$};
\node[vertex] (vertexd) at (3,0)  {$\,$}
	edge [->,>=latex,out=180,in=0,thick] node[auto,pos=0.5]{$\scriptstyle(x,x)$} (vertexe)
	edge [<-,>=latex,out=90,in=270,thick] node[auto,swap,pos=0.5]{$\scriptstyle(y,y)$}(vertexc)
	edge [->,>=latex,out=135,in=315,thick] node[auto,xshift=0.15cm,pos=0.6]{$\scriptstyle(y,y)$} (vertexb);
\node at (1.5,4) {$\scriptstyle aba$};
\node[vertex] (vertexaba) at (1.5,4)  {$\,\,\,\,$}
	edge [->,>=latex,out=215,in=35,thick] node[auto,swap,xshift=0.3cm,pos=0.7]{$\scriptstyle(y,y)$}(vertexa)
	edge [->,>=latex,out=325,in=145,thick] node[auto,xshift=-0.3cm,pos=0.7]{$\scriptstyle(x,x)$}(vertexc);
\node at (-1,4) {$\scriptstyle cadac$};
\node[vertex] (vertexdcadac) at (-1,4)  {$\quad\,\,\,$}
	edge [->,>=latex,out=200,in=130,thick] node[auto,swap,pos=0.5,xshift=0.0cm]{$\scriptstyle(y,y)$} (vertexe)
	edge [->,>=latex,out=0,in=180,thick] node[auto,pos=0.5]{$\scriptstyle(x,x)$}(vertexaba);
\end{tikzpicture}
\end{center}
We need to count the paths from $cadac$ to $e$ in this diagram. They go either straight to $e$, or straight to $c$. Using the formula in \eqref{formFc} for $F^l_c$, we have
\[
|F_{cadac}^k|=  2^{k-1}+ |F_c^{k-2}|=
 2^{k-1}+\frac{2^{k-1}-2^{k-(3j+4)}}{7}\quad\text{where $3j+4\leq k\leq 3j+6$,}
\]
and hence
\[
|X|^{-k}|F_{cadac}^k|=2^{-1}+\frac{2^{-1}-2^{-(3j+4)}}{7}\quad\text{where $3j+4\leq k\leq 3j+6$.}
\]
Thus $\phi(u_{cadac})=c_{cadac}=\lim_{k\to \infty}|X|^{-k}|F_{cadac}^k|=\frac{4}{7}$.
\end{example}


\end{document}